\documentclass[preprint,12pt]{elsarticle}

\usepackage{amssymb}
\usepackage{amsmath}
 \usepackage{amsthm}
 \usepackage{graphicx}
 \usepackage{subfigure}
 \usepackage{verbatim}
 \usepackage{color}

 \usepackage{lineno}

\renewcommand{\b}{\boldsymbol}
\newtheorem{myRem}{Remark}

\newtheorem{proposition}{Proposition}
\newtheorem{theorem}{Theorem}

\begin{document}

\begin{frontmatter}



\title{A Moving Mesh Method for Porous Medium Equation by the Onsager Variational Principle}


\author[1]{Si Xiao }
\ead{xiaosi@lsec.cc.ac.cn}
\author[1]{Xianmin Xu \corref{cor1}}
\ead{xmxu@lsec.cc.ac.cn}

\cortext[cor1]{Corresponding author}

\cortext[fund]{Funding: The work was partially supported by NSFC 11971469 and 12371415.}

\affiliation[1]{organization={LSEC, ICMSEC, NCMIS, Academy of Mathematics and Systems Science, Chinese Academy of Sicences},
            city={Beijing},
            postcode={100190}, 
            country={China}}

\begin{abstract}

In this paper, we introduce a new approach to solving the porous medium equation using a moving mesh finite element method that leverages the Onsager variational principle as an approximation tool. Both the continuous and discrete problems are formulated based on the Onsager principle. The energy dissipation structure is maintained in the semi-discrete and fully implicit discrete schemes.  We also develop a fully decoupled explicit scheme by which only a few linear equations are solved sequentially in each time step. The numerical schemes exhibit an optimal convergence rate when the initial mesh is appropriately selected to ensure accurate approximation of the initial data. Furthermore, the method naturally captures the waiting time phenomena without requiring any manual intervention.
\end{abstract}

%


\end{frontmatter}


\section{Introduction}
The porous medium equation (PME) serves as a prominent mathematical model frequently utilized to comprehensively describe various physical and biological phenomena, including gas flow, nonlinear heat transport, groundwater movement, etc. The PME is a nonlinear partial differential equation taking the form:
\begin{equation}
\label{e:PME} \partial_t \rho = \Delta \rho^m, \quad m>1. 
\end{equation} 
Extensive research has revealed numerous intriguing properties of this equation. One particularly fascinating characteristic is its association with a finite speed of propagation. This stands in stark contrast to the linear heat equation ($m=1$), where heat propagation speed is infinite. Notably, if the initial value of 
$\rho$ possesses compact support, the boundary of this support moves at a finite velocity even when the equation is defined across the entire space, leading to a typical free boundary problem. Furthermore, under specific initial conditions, the solution of the PME may demonstrate a waiting time phenomenon \cite{knerr1977porous}, wherein the free boundary remains stationary until a critical time threshold is surpassed. These properties have been extensively investigated in mathematics (e.g in  \cite{oleinik1958cauchy,kalavsnikov1967formation,aronson1983initially,herrero1987one,aronson2006porous,vazquez2007porous}). 


From a numerical perspective, solving the PME presents several challenges. Firstly, the free boundary of the solution is not easy to captured by standard numerical methods. The moving velocity of the boundary depends on the derivative of the solution.To ensure precise velocity calculations, it is imperative to compute the gradient velocity with precision. Secondly, singularities in the solution of the PME arise at the free boundary. For larger $m$, the regularity is worse.  Thirdly, it is also a  challenge to accurately compute the waiting time.


These challenges have sparked significant interest in the PME, leading to the development of numerous numerical methods in literature. In early investigations, researchers commonly transformed the PME into an alternative equation based on pressure, which exhibits better regularity. The free boundary was typically addressed using either a front tracing method or by extending the equation to a larger fixed domain to avoid boundary issues. The pressure equation was often solved using a  finite element method \cite{nochetto1988approximation} or a finite difference method \cite{graveleau1971finite, tomoeda1983numerical, dibenedetto1984interface, monsaingeon2016explicit}, yielding a convergence rate typically of first order with respect to the mesh size.
Later on, the original PME was solved by using a high order local discontinuous Galerkin method to improve the accuracy. A suitable non-negativity preserving limiter was used to prevent oscillations near the free interface \cite{zhang2009numerical}. More recently, new numerical schemes have emerged for the PME, focusing on preserving the positivity of density and the energy dissipation relation \cite{gu2020bound, vijaywargiya2023two}.

 
 Given the singularity and moving boundary present in the solution of the PME, a natural approach is to employ  moving mesh  methods \cite{huang1994moving,li2001moving,tang2005moving,budd2009adaptivity,hu2011simulating} to address these complexities. Baines et al. introduced a moving mesh finite element method, where the mesh dynamically adjusts based on a local scale-invariant conservation principle \cite{baines2005moving, baines2005moving2, baines2011velocity}. Numerical experiments demonstrate a second-order convergence rate for this method when the mesh is appropriately selected initially. 
 Duque et al. utilized a moving mesh partial differential equation (MMPDE) method to tackle the PME with a variable component  $m$ \cite{duque2015application, duque2014numerical}. Ngo \& Huang \cite{ngo2017study, ngo2019adaptive} also applied the MMPDE approach to solve both the PME and the transformed pressure equation. Their detailed study on the impact of the metric tensor in a novel implementation of the velocity equation revealed that Hessian-based adaptive meshes yield optimal second-order convergence. 
 Theoretical analysis for the moving mesh methods is usually difficult due to a lack of a variational formulation.


The previously mentioned methods are  based on an Eulerian framework, where the computational domain is in  the physical domain. Recently, there has been a surge in the development of Lagrangian-type methods, where the problem is formulated within a reference domain. Liu and his collaborators have introduced several schemes based on the energy variational method \cite{duan2019numerical, liu2020lagrangian, duan2022second}. Their research highlighted that specific selections of the energy and dissipation functionals can significantly enhance the efficiency of the numerical method. Furthermore, optimal error estimates have been established  under certain regularity assumptions \cite{duan2022second}. Additionally, Carrillo et al. have devised a Lagrange-type scheme \cite{carrillo2018lagrangian} and a particle method \cite{carrillo2015numerical} utilizing the theory of Wasserstein gradient flow. One  challenge with Lagrangian-type methods is to solve some highly nonlinear problems at each time step. Proper iterative solver has to be  chosen to solve the problem efficiently.


In this paper, our objective is to devise a numerical method for the PME that combines the strengths of both the conventional moving mesh method (e.g. in an Eulerian framework) and Lagrangian-type methods (e.g. based on a variational formulation). To achieve this goal, we will leverage the Onsager principle as an approximation tool. The Onsager variational principle is a fundamental law for characterizing irreversible processes in nonequilibrium thermodynamics \cite{onsager1931reciprocal, onsager1931reciprocal2}. This principle has been instrumental in deriving mathematical models for various soft matter physics problems \cite{DoiSoft}. Recent studies have demonstrated the effectiveness of the Onsager principle as a powerful approximation tool for deriving reduced models (c.f. \cite{Doi2015, XuXianmin2016, ManXingkun2016, doi2019application, zhang2022effective}, among others). Moreover, the Onsager principle has been utilized in designing numerical schemes \cite{lu2021efficient, xu2023variational}. It has also been shown that the moving finite element method can be naturally derived from this principle \cite{xu2023variational}. 


Motivated by the previous studies, we derive a novel numerical method for the PME in this paper. We first demonstrate the natural derivation of the PME from the Onsager principle. We treat the continuum equation
$\partial_t  \rho + \nabla\cdot (\rho {v})=0$ as a constraint and incorporate a Lagrange multiplier into the Reyleignian functional to obtain a mixed version of the PME. Subsequently, we approximate the unknown function $\rho$ and the multiplier by finite element functions on a triangulation with movable nodes. These nodes are assumed to move at the same velocity as dictated by the continuum equation. By reapplying the Onsager principle, we formulate a semi-discrete numerical scheme in mixed form, demonstrating that the scheme upholds the same energy dissipation relation as the continuous problem. When employing an explicit time derivative discretization, a decoupled numerical scheme emerges, where a few linear equations are sequentially solved at each time step. We also develop an implicit scheme which leads to the establishment of the energy dissipation relation for the fully discrete method. We also explore the impact of the mass conservation property by varying the boundary condition for the Lagrangian multiplier. Numerical examples show that  our methods exhibit optimal convergence rates and accurately capture the waiting time phenomenon in both one-dimensional and two-dimensional scenarios.

The rest of the paper is organized as follows. In Section 2, we present the derivation of the PME by using the Onsager variational principle. We then apply the Onsager principle in the finite element space to derive two numerical schemes in Section 3. 
In Section 4, we discuss briefly an alternative semi-discrete scheme with improved mass conservation property. Numerical examples are given in Section 5 to illustrate the efficiency of our methods. A few concluding remarks are given in the last section.

\section{Derivation of the PME by the Onsager principle}
In this section, we will derive a model for gas flow in a homogeneous porous medium by the Onsager variational principle\cite{onsager1931reciprocal2,DoiSoft}. For simplicity, we consider the dimensionless model throughout the paper. We denote by $\rho(x,t)$ the mass density of a gas in a porous medium. 
Suppose the free energy is given by \cite{otto2001geometry}
\begin{equation}
	\mathcal{E}(\rho) = \int_{\Omega(t)} f(\rho(x,t)) dx,
\end{equation}
where $\Omega(t) \subset \mathbb{R}^d$ is the domain where the gas flow occupies at time $t$,  and $f(\rho) = \frac{1}{m-1}\rho^m (m>1)$ is the free energy density. From thermodynamic theory, the gas pressure is $p = \delta \mathcal{E} / \delta \rho = f'(\rho)$. 
Let the vector field $v(x,t):\Omega\times[0,T]\mapsto \mathbb{R}^n$ be the average velocity of the gas. Denote by $J = \rho v$ the mass flux in the system. 
Then we have the mass conservation equation, 
\begin{equation}
	\partial_t \rho + \nabla \cdot J = 0,
	\label{eq:continuity eq}
\end{equation}
In this setting, the mass flux through the boundary $\partial \Omega$ is zero, i.e., $J\cdot n|_{\partial \Omega} = 0$, where $n$ is the outward unit normal on the boundary.

To derive the PME by the Onsager principle, we first calculate the changing rate  of the free energy. By  Reynold's transport theorem, we obtain 
\begin{equation}
	\dot{\mathcal{E}} = \frac{d}{dt}\int_{\Omega(t)} f(\rho) dx = \int_{\Omega(t)} f'(\rho)\partial_t \rho dx + \int_{\partial \Omega(t)} f(\rho) v\cdot n ds = \int_{\Omega(t)} f'(\rho)\partial_t \rho dx,
\end{equation}
where the last equality uses the zero mass flux  condition on the boundary. Then we define the dissipation function as 
\begin{equation}
	\Phi(\rho;v) = \int_{\Omega(t)} \frac{1}{2}\rho |v|^2 dx.
\end{equation} 
The Rayleighian functional is given by
\begin{equation}
	\mathcal{R} = \Phi + \dot{\mathcal{E}}.
\end{equation}

There are two ways to derive the PME by the Onsager variational principle. In the first approach, we derive the PME by minimizing the Rayleighian functional with respect to the mass flux $J$.
Direct calculate give
\begin{equation}
	\begin{split}
		\mathcal{R} & = \int_{\Omega(t)}\frac{|J|^2}{2\rho} dx  +  \int_{\Omega(t)} f'(\rho)\partial_t \rho dx = \int_{\Omega(t)}\frac{|J|^2}{2\rho} dx  -  \int_{\Omega(t)} f'(\rho) \nabla \cdot J dx \\
		& = \int_{\Omega(t)}\frac{|J|^2}{2\rho} dx + \int_{\Omega(t)} \nabla f'(\rho) J dx,
	\end{split}
\end{equation}
where the second equation utilizes the continuity equation \eqref{eq:continuity eq} and the last equation utilizes the integration by parts. We minimize $\mathcal{R}$ with respect to the flux $J$, i.e.
\begin{equation*}
	\min_J \mathcal{R}(J).
\end{equation*}
The corresponding Euler-Langrange equation is
\begin{equation}
	J = - \rho \nabla f'(\rho) = -\nabla \rho^m.
\end{equation}
Substituting the equation into \eqref{eq:continuity eq}, we have the PME
\begin{equation}	\label{PME}
	\partial_t \rho = \Delta \rho^m.
\end{equation}

Although the above derivation is straightforward, we will present a different approach below, which is more helpful  to propose a numerical method. We will derive the PME by minimizing the Rayleighian functional with respect to $\partial_t \rho$ and $v$ under the constraint of 
the mass conservation equation. That is to consider the problem
\begin{equation}\label{e:Onsager}
	\begin{aligned}
		&\min _{\partial_t \rho, v}\mathcal{R}(\rho;\partial_t \rho, v) = \Phi(\rho;v) + \dot{\mathcal{E}}(\rho ; \partial_t \rho) = \frac{1}{2}\int_{\Omega(t)} \rho |v|^2 dx + \int_{\Omega(t)} f'(\rho)\partial_t \rho dx, \\
		&s.t. \quad \partial_t \rho + \nabla \cdot (\rho v) = 0.
	\end{aligned}
\end{equation}
By introducing a Lagrange multiplier $\lambda(x)$, we obtain a Lagrangian functional
\begin{equation}
	\tilde{\mathcal{R}}= \frac{1}{2}\int_{\Omega(t)} \rho |v|^2 dx + \int_{\Omega(t)} f'(\rho)\partial_t \rho dx - \int_{\Omega(t)} \lambda(x) (\partial_t \rho + \nabla \cdot (\rho v)) dx.
\end{equation}
The corresponding Euler-Lagrange equation is
\begin{equation}
	\begin{cases}
		f'(\rho) - \lambda(x) = 0 ,\\
		 v +  \nabla \lambda = 0,\\
		\partial_t \rho + \nabla \cdot (\rho v) = 0.
	\end{cases}
	\label{eq:EL}
\end{equation} 
In the above equation,  the Lagrange multiplier $\lambda(x)$ has a physical interpretation
that $\lambda = f'(\rho)=\frac{m}{m-1}\rho^{m-1}$ is the pressure. Since $m>1$, we have $\lambda=0$ on $\partial\Omega$.
The property is important for us to derive a decoupled scheme in the next section.
 Note that $v = -\nabla \lambda $ is the Darcy's law. 
By the above equation, we can easily derive the PME \eqref{PME}.
The equation \eqref{eq:EL} can be seen as a mixed form of  the equation \eqref{PME}.

It is easy to see that the solution $\rho$ of the equation \eqref{eq:EL} satisfies the following energy decay property
\begin{equation}
	\frac{d\mathcal{E}}{dt} \leq 0.
\end{equation}
Actually, by using the property $\rho \geq 0$, we have
\begin{equation}
	\begin{split}
		\frac{d\mathcal{E}}{dt} &= \int_{\Omega(t)} f'(\rho) \partial_t \rho dx = \int_{\Omega(t)} \nabla f'(\rho) \cdot \rho v dx = \int_{\Omega(t)} \nabla \lambda \cdot \rho v dx\\
		& = -\int_{\Omega(t)} \rho|v|^2  dx = -2\Phi(\rho;v) \leq 0.
	\end{split} 
\end{equation}

In the previous derivation, we  see that $v=-\nabla\lambda$ may not be equal to zero on $\partial\Omega$.
This implies that the PME is a free boundary problem. We rewrite the equation in a closed form that
\begin{equation}\label{e:PMEfull}
	\begin{cases}
		\partial_t \rho  = \Delta \rho^m, &\text{in } \Omega(t), \\
		\rho = 0, & \text{on } \partial \Omega(t), \\
		v_n = -\frac{m}{m-1}\nabla \rho^{m-1} \cdot n, &\text{on } \partial \Omega(t),\\
				\rho(x,0) = \rho_0(x) , &  \text{at } t = 0.
	\end{cases}
\end{equation}
where $v_n=v\cdot n$ is the outer normal velocity of the free boundary.
The well-posedness of the equation can be found in \cite{vazquez2007porous}. In addition to the mass conservation and energy decay properties, the PME has some other interesting properties, like the waiting time phenomenon and the finite diffusion velocity, etc. 
In next section, we will derive a numerical method to the PME by using the Onsager principle as an approximation tool.

\section{A moving mesh finite element method}
For simplicity in presentation, we derive a moving mesh finite element method by using the Onsager variational principle in one dimension in this section. The derivation can be generalized to higher dimensional cases straightforwardly (c.f. the two dimensional cases in the appendix). Let the interval $I(t) = [a(t),b(t)]$ be the domain where the PME is defined.  The equation~\eqref{e:PMEfull} is reduced to
\begin{equation}
	\begin{cases}
		\partial_t \rho(x,t) = \partial_{xx} \rho^m(x,t), &  x\in I(t), t>0,\\
		\rho(a(t),t) = 0, \quad \rho(b(t),t) = 0,& t>0,\\
				\dot{a} = -\frac{m}{m-1} \partial_x \rho^{m-1}(a) ,\quad \dot{b} = -\frac{m}{m-1} \partial_x \rho^{m-1}(b), & t>0,\\
		\rho(x,0) = \rho^0(x),  & x\in I(0). \\
	\end{cases}
	\label{eq:1D PME}
\end{equation}
We will not directly discretize the problem \eqref{eq:1D PME}. Instead we derive a discrete problem by the Onsager variational principle.

\subsection{Semi-discretization}
We first partition the interval $I(t)$ by $N+1$  knots,
\begin{equation}
	X(t):=\{a(t) = x_0(t) < x_1(t) <\cdots < x_N(t) = b(t)\}.
\end{equation}
Notice that the knots may change positions with respect to time. Denote the partition as $\mathcal{T}_h: = \{ I_i \}_{i=1}^N$, where $I_i = (x_{i-1}(t),x_i(t)]$. Then we can define the finite element space $V_h^t$
\begin{equation}
	V_h^t : = \{ u_h \in C(I(t)):u_h \text{ is linear in } I_i(t),\forall i = 1,...,N \}.
\end{equation}
Denote by $V_{h,0}^t =\{u_h \in V_h^t: u_h(a) =u_h(b) =0 \}$. For any function $\rho_h(x,t) \in V_{h,0}^t$, it can be written as 
\begin{equation}
	\rho_h(x,t) = \sum_{i=1}^{N-1} \rho_i(t) \phi_i(x,t),
\end{equation}
where $\phi_i(x,t)$ is the finite element basis function associated with $x_i$, i.e.
\begin{equation}
	\phi_i(x,t) = \phi_i^l + \phi_i^r = \frac{x-x_{i-1}(t)}{x_i(t)-x_{i-1}(t)} \chi _{I_i}(x) + \frac{x_{i+1}(t) - x}{x_{i+1} -x_i(t)}\chi_{I_{i+1}}(x),
\end{equation}
where $\chi_{I_i}$ is the characteristic function corresponding to $I_i$.
Due to the Dirichlet boundary condition $\rho_0(t) = \rho_N(t) = 0$, there are $2N$ time dependent parameters in the formula of $\rho_h(t,x)$, i.e.,
\begin{equation}
	\{ \rho_1(t), \rho_2(t),\cdots,\rho_{N-1}(t),x_0(t),x_1(t),\cdots,x_N(t)\}.
	\label{test}
\end{equation}
We aim to approximate the solution $\rho$ of the  problem \eqref{eq:1D PME} by a discrete function $\rho_h(t,x)$. For that propose, we will derive a dynamic equation for $\rho_i(t)$ and $x_i(t)$ by using the Onsager principle.

Firstly, we discretize the  energy functional and as follows. Notice that the time derivative and space derivative of $\rho_h(t,x)$ are
respectively given by
\begin{align*}
	\partial_t \rho_h  &= \sum_{i=1}^{N-1}\dot{\rho}_i(t) \phi_i(x,t) + \sum_{i=0}^{N}\dot{x}_i(t) \psi_i(x,t),\\
	\partial_x \rho_h &= \sum_{i=1}^{N-1}\rho_i(t) \partial_x \phi_i(x,t),
\end{align*}
where 
\begin{equation*}
	\psi_i(x,t) = \frac{\partial \rho_h}{\partial x_i} = -D_h\rho_{i-1}\phi_i^l - D_h \rho_{i} \phi_i^r,
\end{equation*} 
with $D_h \rho_i = \frac{\rho_{i+1}(t)-\rho_{i}(t)}{x_{i+1}(t)-x_{i}(t)}$.

Denote by $\b{\rho} = (\rho_1(t),...,\rho_{N-1}(t))^{T}, \b{x} = (x_0(t),...,x_N(t))^T$. The discrete energy functional $\mathcal{E}$ with respect to $\rho_h$ is given by
\begin{equation}
	\mathcal{E}_h(\b{\rho},\b{x}) = \sum_{i=1}^{N} \int_{I_i} f(\rho_h) dx .
\end{equation}
Then the changing rate of the discrete energy is calculated as
\begin{equation}
		\dot{\mathcal{E}_h} (\b{\rho},\b{x};\dot{\b{\rho}},\dot{\b{x}})= \sum_{i=1}^{N-1} \frac{\partial \mathcal{E}_h}{\partial \rho_i}\dot{\rho}_i + \sum_{i=0}^N \frac{\partial \mathcal{E}_h} {\partial x_i}\dot{x}_i,
	\label{discrete energy}
\end{equation}
where
\begin{align*}
	\frac{\partial \mathcal{E}_h}{\partial \rho_i} &= \int_I f'(\rho_h) \phi_i dx,\quad i=1,...,N-1;\\
	\frac{\partial \mathcal{E}_h} {\partial x_i} &= \int_I f'(\rho_h)\psi_i dx, \quad i=0,...,N.
\end{align*}


In order to obtain the discrete dissipation functional, we need to discretize the velocity $v(x,t)$. We use a piecewise linear function $v_h(x,t) =  \sum_{i=0}^N v_i(t)\phi_i(x,t)$ in $V_h^t$ to approximate the velocity $v(x,t)$. Denote by $\b{v} = (v_0(t),...,v_N(t))^T$. Then we calculate the discrete dissipation function $\Phi_h$ as
\begin{equation}\label{eq:discDissp1}
		\Phi_h(\b{\rho},\b{x};\b{v}) = \sum_{i=1}^{N} \int_{I_i} \frac{1}{2} \rho_h(x,t) v_h(x,t)^2 dx.
\end{equation}
Suppose that the mesh knots move with velocity $v_h$ in a Lagrange manner, i.e.
\begin{equation}
	\dot{x}_i(t) = v_h(x_i,t),\qquad i=0,...,N.
	\label{scheme1_1}
\end{equation}
Then the time derivative of $\mathcal{E}_h$ can be rewritten as
\begin{equation}\label{eq:discEnergyRate1}
	\dot{\mathcal{E}}_h (\b{\rho},\b{x}; \dot{\b{\rho}},\b{v}) =  \sum_{i=1}^{N-1} \frac{\partial \mathcal{E}_h}{\partial \rho_i}\dot{\rho}_i + \sum_{i=0}^N \frac{\partial \mathcal{E}_h} {\partial x_i}v_i.
\end{equation}
The discrete Rayleighian functional is defined as
\begin{equation*}
	\mathcal{R}_h(\b{\rho},\b{x}; \dot{\b{\rho}},\b{v}) = \Phi_h(\b{\rho};\b{x},\b{v}) + 	\dot{\mathcal{E}}_h (\b{\rho},\b{x}; \dot{\b{\rho}},\b{v}).
\end{equation*}
By the Onsager variational principle, $(\dot{\b{\rho}},\b{v})$ is obtained by
\begin{equation}
	\begin{aligned}
		&\min_{\dot{\b{\rho}},\b{v}} 	\mathcal{R}_h(\b{\rho},\b{x}; \dot{\b{\rho}},\b{v}) \\
		 & s.t. \quad \int_{I} (\partial_t \rho_h + \partial_x(\rho_h v_h) ) w_h dx = 0, \quad \forall w_h \in V_{h,0}^t.
	\end{aligned}		\label{eq:ConstraintP}
\end{equation}

To deal with the constraint in the above problem, we introduce a discrete Lagrangian multiplier $\lambda_h= \sum_{i=1}^{N-1} \lambda_i \phi_i(x,t)$. By integration by part, we have
\begin{equation*}
	\int_{I} (\partial_t \rho_h + \partial_x(\rho_h v_h) ) \lambda_h dx = \int_{I} (\partial_t \rho_h  \lambda_h - \rho_h v_h  \partial_x{\lambda_h}) dx.
\end{equation*}
Then the discrete Lagrangian functional is given by
\begin{equation}
		\tilde{\mathcal{R}}_h = \Phi_h + \dot{\mathcal{E}}_h - \int_{I} (\partial_t \rho_h \lambda_h - \rho_h v_h \partial_x \lambda_h) dx.
	\label{eq:DRay1}
\end{equation}
Notice that the problem can be seen as a discrete version of the equation~\eqref{e:Onsager}. Here we consider a weak form of
the continuum equation in the constraint. Notice that the test function is chosen to be in a finite element space $V_{h,0}^t$ instead of $V_h^t$. This will lead to a discrete multiplier $\lambda_h\in V_{h,0}^t$ in the Euler-Lagrange equation. This is consistent with
the continuous problem where the multiplier(pressure) $\lambda=0$ on $\partial\Omega$.

The Euler-Lagrange equation  corresponding to the problem \eqref{eq:DRay1} is given by
\begin{equation}
	\begin{cases}
 \frac{\partial \mathcal
		{E}_h}{\partial \rho_i}  -  \int_I \phi_i\lambda_hdx=0,& i=1,...,N-1 ;\\
	\int_I \rho_h  v_h \phi_i dx + \frac{\partial \mathcal{E}_h}{\partial x_i} - \int_I \psi_i\lambda_h dx +\int_I \rho_h \partial_x\lambda_h\phi_i dx =0,& i=0,...,N;\\
	 \int_I \partial_t \rho_h \phi_i dx - \int_I \rho_h v_h \partial_x\phi_i dx=0 ,& i=1,...,N-1.
	 \end{cases}
 \label{eq:discreteEL}
\end{equation}
Notice that $\dot{\b{x}} = \b{v}$, the equations \eqref{eq:discreteEL} can be written in an algebraic form
\begin{equation}
	\begin{cases}
	\b{M}(\b{x}(t))\b{\lambda}(t) = \frac{\partial \mathcal
		{E}_h}{\partial \b{\rho}}(\b{x}(t),\b{\rho}(t)),\\
			\b{D}(\b{x}(t),\b{\rho}(t))\dot{\b{x}}(t) =- \frac{\partial \mathcal
				{E}_h}{\partial \b{x}}(\b{x}(t),\b{\rho}(t)) + \Big(\b{B}^T(\b{x}(t)) - \b{E}^T(\b{x}(t),\b{\rho}(t)) \Big) \b{\lambda}(t) ,\\
			\b{M}(\b{x}(t))\dot{\b{\rho}}(t) + \Big(\b{B}(\b{x}(t)) - \b{E}(\b{x}(t),\b{\rho}(t)) \Big)\dot{\b{x}}(t)=0,
	\end{cases}
\label{eq:matrixEL}
\end{equation}
where  $\b{M}\in \mathbb{R}^{N-1,N-1}, \b{D}\in \mathbb{R}^{N+1,N+1}, \b{B}\in \mathbb{R}^{N-1,N+1}, \b{E}\in \mathbb{R}^{N-1,N+1},$ such that
\begin{align*}
	M_{ij}(\b{x}(t)) &= \int_{I(t)} \phi_i\phi_j dx ; \quad	D_{ij}(\b{x}(t),\b{\rho}(t)) = \int_{I(t)} \rho_h \phi_i \phi_j dx; \\
	B_{ij}(\b{x}(t)) &= \int_{I(t)} \phi_i \psi_j dx ;\quad
	E_{ij}(\b{x}(t),\b{\rho}(t)) = \int_{I(t)} \rho_h \partial_x \phi_i \phi_j dx.
\end{align*}


In the following of the subsection, we give some important properties for the differential-algebraic system \eqref{eq:matrixEL}. We first address the existence of solutions $\b{\rho}(t)$ and $\b{x}(t)$ with non-negative initial data $\b{\rho}_0$ and a initial partition $\b{x}_0$. We need some  assumptions:
\begin{itemize}
	\item[(\textbf{A1}).] The intervals $I_i(t),i = 1,...,N$ are well-defined, i.e., $x_{i-1}(t)<x_i(t)$.
	\item[(\textbf{A2}).] The discrete density function $\b{\rho}(t)$ is non-negative for all $t$.
\end{itemize}
\begin{proposition}
	Under the assumptions (A1) and (A2), there exist an unique solution for differential-algebraic system \eqref{eq:discreteEL}.
\end{proposition}
\begin{proof}
	By the assumption (A1), the mass matrix $M(\b{x}(t))$ is positive definite for any $t$. Then $\b{\lambda}(t)$ can be solved in the algebraic equation, i.e. $\b{\lambda} = M^{-1} \frac{\partial \mathcal{E}_h}{\partial \b{\rho}}$. By the assumption (A2), the matrix $D(\b{x},\b{\rho})$ is also positive definite. Then the differential-algebraic system reduces to a system of ordinary differential equations (ODEs)
	\begin{equation*}
		\begin{cases}
			\dot{\b{x}}(t) = - g_1(\b{x}(t),\b{\rho}(t)) ,\\
			\dot{\b{\rho}}(t) = g_2(\b{x}(t),\b{\rho}(t))g_1(\b{x}(t),\b{\rho}(t)),
		\end{cases}
	\end{equation*}
	where 
	\begin{align*}
	&	g_1(\b{x}(t),\b{\rho}(t))\\
	 &\quad= \b{D}^{-1}(\b{x}(t),\b{\rho}(t))\Big[\frac{\partial \mathcal
			{E}_h}{\partial \b{x}}(\b{x}(t),\b{\rho}(t)) - \Big(\b{B}^T(\b{x}(t)) - \b{E}^T(\b{x}(t),\b{\rho}(t)) \Big) \b{\lambda}(t)\Big],\\
		&g_2(\b{x}(t),\b{\rho}(t)) = \b{M}^{-1}(\b{x}(t)) \Big(\b{B}(\b{x}(t)) - \b{E}(\b{x}(t),\b{\rho}(t)\Big).
	\end{align*}
	It is easy to verify that the vector-valued functions $g_1(\b{x},\b{\rho})$ and
	 $g_2(\b{x},\b{\rho})$ have continuous partial derivatives with respect to $\b{\rho}$ and $\b{x}$ on a bounded closed convex domain by direct calculations. Therefore they are Lipschitz continuous with respect to $(\b{x}(t),\b{\rho}(t))$. By the Picard-Lindelof theorem, we know that the ODE system has a unique solution for given proper initial values.
\end{proof}
The following theorem states the discrete energy dissipation relations.
\begin{theorem}
	Under the assumptions (A1) and (A2), and let $\b{\rho}(t),\b{x}(t)$ be the solution of the equations \eqref{scheme1_1} and \eqref{eq:discreteEL}. Let $\rho_h(t,x) \in V_{h,0}^t$ be the corresponding discrete density function and $v_h \in V_h^t$ be the discrete velocity function. Then we have
	\begin{equation}
		\frac{\partial \mathcal{E}_h(\rho_h) }{\partial t} = -2\Phi_h(\rho_h,v_h)  \leq 0.
	\end{equation}
\end{theorem}
\begin{proof}
	The proof is given by  straightforward calculations
	\begin{equation*}
		\begin{split}
			&\frac{\partial \mathcal{E}_h(\rho_h) }{\partial t} = \sum_{i=1}^{N-1} \frac{\partial \mathcal{E}_h}{\partial \rho_i}\dot{\rho}_i + \sum_{i=0}^N \frac{\partial \mathcal{E}_h} {\partial x_i}\dot{x}_i  \\
			&=	 \sum_{i=1}^{N-1} \int_I \phi_i \lambda_h dx \dot{\rho_i}	+  \sum_{i=0}^{N}\Big( -\int_I \rho_h \partial_x \lambda_h \phi_i dx -\int_I \rho_h v_h \phi_i dx +\int_I \psi_i \lambda_h dx\Big) v_i \\
			&= \sum_{i=1}^{N-1} \Big(\int_I \partial_t \rho_h \phi_i dx - \int_I \rho_h v_h \partial_x \phi_i dx \Big)\lambda_i - \sum_{i=0}^{N}\int_I \rho_h v_h \phi_i dx v_i \\
			&= -\int_I \rho_h v_h^2 dx = -2\Phi_h(\rho_h,v_h) .
		\end{split}
	\end{equation*}
	By the assumption (A2), the function $\rho_h(x,t) \geq 0$ for all $x \in I(t)$. Then we can easily see that
	\begin{equation*}
		\int_I \rho_h v_h^2 dx \geq 0.
	\end{equation*}
This leads to the proof of the proposition.
\end{proof}

Finally, we show a property of the semi-discrete problem which is related to the mass conservation.
\begin{proposition}\label{prop:MassConserv0}
	Under the assumptions (A1) and (A2), and let $\b{\rho}(t),\b{x}(t)$ be the solution of the equations \eqref{scheme1_1} and \eqref{eq:discreteEL}. We have the following relations,
	\begin{equation}
		\frac{d}{dt} \int_{I(t)} \rho_h(x,t) \phi_i(x,t) dx =0 ,\quad i=1,...,N-1.
		\label{eq:mass_cons}
	\end{equation}
\end{proposition}
\begin{proof}
	Using integration by parts, we can obtain
	\begin{align*}
		&\frac{d}{dt}\int_{I(t)} \rho_h(x,t) \phi_i(x,t) dx \\
		&= \int_{I_i(t) \cup I_{i+1}(t)} (\partial_t \rho_h + \partial_x(\rho_h v_h))\phi_i + \rho_h (\partial_t \phi_i + v_h\partial_x \phi_i) dx .
	\end{align*}
	Notice that
	\begin{equation*}
		\partial_t \phi_i = \sum_{j=i-1}^{i+1} \partial_{x_j} \phi_i \dot{x_j}  = \sum_{j=i-1}^{i+1} (-\phi_j\partial_x \phi_i) \dot{x_j}  = -(\sum_{j=i-1}^{i+1} \phi_j \dot{x_j} )\partial_x \phi_i = -v_h \partial_x \phi_i.
	\end{equation*}
	This leads to $\partial_t \phi_i + v_h \partial_x \phi_i = 0$. Thus by \eqref{eq:discreteEL} we have
	\begin{equation*}
		\frac{d}{dt}\int_{I(t)} \rho_h(x,t) \phi_i(x,t) dx = 0,\quad i=1,...,N-1.
	\end{equation*}
\end{proof}

\begin{myRem}\label{remark1}
	We can rewrite the equations \eqref{eq:mass_cons} in a vector form
	\begin{equation*}
		\frac{d}{dt} \Big(\b{M}(t)\b{\rho}(t)\Big)  = 0,
	\end{equation*}
	Integrating the equation from $0$ to $T$, we have $\b{M}(T)\b{\rho}(T) = \b{M}(0)\b{\rho}(0)$.  Notice that this does not imply the exact mass conservation property. We will give more discussions on this issue in the next section. 
	In addition,
	given a discrete initial value $\b{\rho}(0)$, it is easy to show $\b{M}(0)\b{\rho}(0) \geq 0$. However, since the mass matrix $\b{M}(T)$ is not a M-matrix in general, we can not prove the positivity of $\b{\rho}(T)$. That is why we need the assumption (A2). To guarantee the positivity of $\b{\rho}(T)$, one can use the lumped mass method. This is beyond 
	the scope of our studies presented in this paper. 
\end{myRem}

\subsection{Full discretization}
In order to get a fully discrete numerical scheme, we  introduce a proper temporal discretization to the semi-discrete equations \eqref{scheme1_1} and \eqref{eq:discreteEL}. Let the time step be $\tau$, then we set the solution at $t = t^n$
as  $\rho_i^n = \rho_i(t^n),x_i^n=x_i(t^n)$, and the solution at $t^{n+1}$ as $\rho_i^{n+1}, x_i^{n+1}$. We define the finite difference operator $\bar{\partial}\rho^n = (\rho^{n+1}-\rho^n)/\tau$, then let $v_i^{n+1} :=\bar{\partial}x_i^n = (x_i^{n+1}-x_i^n)/\tau$. We first consider {\it an explicit Euler scheme} as follows

\begin{align}
	\int_{I^n} \phi_i^n \lambda_h^{n+1}dx&=\frac{\partial \mathcal
		{E}_h^n}{\partial \rho_i^n} ,\qquad\qquad\qquad\qquad  i=1,...,N-1; \label{eq:ExEuler1}\\
	\int_{I^n} \rho_h^n v_h^{n+1} \phi_i^n dx &= -\frac{\partial \mathcal{E}_h^n}{\partial x_i^n} + \int_{I^n} \psi_i^n\lambda_h^{n+1} dx -\int_{I^n} \rho_h^n \partial_x \lambda_h^{n+1} \phi_i^n dx ,\nonumber \\
	&\qquad\qquad\qquad\qquad\qquad\qquad i=0,...,N;  \label{eq:ExEuler2}
			 	\end{align}
			 			\begin{align}
	\int_{I^n} \Big(\sum_{j=1}^{N-1} \bar{\partial}{\rho}_j^n\phi_j^n \Big) \phi_i^n &=- \int_{I^n} \Big(\sum_{j=0}^{N} v_j^{n+1} \psi_j^n \Big) \phi_i^n dx + \int_{I^n} \rho_h^n v_h^{n+1} \partial_x\phi_i^n dx ,\nonumber\\
	&\qquad\qquad\qquad\qquad\qquad\qquad i=1,...,N-1. \label{eq:ExEuler3}
\end{align}
The above scheme is decoupled and easy to implement. We need only to solve a few linear
equation successively in each time step. This is efficient in general case. A drawback of the scheme is that
we show choose a small time step $\tau$   to guarantee  numerical stability. 

Another approach is to adopt  {\it an implicit  numerical scheme} for the temporal discretization as follows,
\begin{align}
	\int_{I^n} \phi_i^n \lambda_h^{n+1} dx &=\frac{\partial \mathcal
		{E}_h^{n+1}}{\partial \rho_i^{n+1}},\qquad\qquad\qquad\qquad  i=1,...,N-1;  \label{scheme2_2}\\
	\int_{I^n} \rho_h^n v_h^{n+1} \phi_i^n dx &=- \frac{\partial \mathcal{E}_h^n}{\partial x_i^n} + \int_{I^n} \psi_i^n \lambda_h^{n+1} dx -\int_{I^n} \rho_h^n \partial_x \lambda_h^{n+1} \phi_i^n dx ,\nonumber \\
		&\qquad\qquad\qquad\qquad\qquad\qquad
	 i=0,...,N; \label{scheme2_3} \\
	\int_{I^n} \Big(\sum_{j=1}^{N-1} \bar{\partial}{\rho}_j^n\phi_j^n \Big) \phi_i^n &=- \int_{I^n} \Big (\sum_{j=0}^N v_j^{n+1} \psi_j^n \Big)  \phi_i^n dx + \int_{I^n} \rho_h^n v_h^{n+1} \partial_x\phi_i^n dx ,\nonumber \\
		&\qquad\qquad\qquad\qquad\qquad\qquad i=1,...,N-1.\label{scheme2_4}
\end{align}
In the following theorem, we prove the energy stability of  the implicit scheme.

\begin{theorem}
	Let $\rho_i^n,x_i^n$ and $\rho_i^{n+1},x_i^{n+1}$ be the solutions of the equations  \eqref{scheme2_2}-\eqref{scheme2_4} at $t = t^n$ and $t=t^{n+1}$, respectively. We have the following result:
	\begin{equation}
		\mathcal{E}_h(\b{\rho}^{n+1},\b{x}^{n+1}) \leq \mathcal{E}_h(\b{\rho}^{n},\b{x}^n) .
	\end{equation}
\end{theorem}

\begin{proof}
	Denote by $r_i^n(s) =x_{i-1}^n+s(x_i^n-x_{i-1}^n)$, $\rho_h\circ r_i^n  = \rho_{i-1}^n + s(\rho_i^n - \rho_{i-1}^n)$.
	Due to the convexity of function $f(\rho)$, we have the following property
	\begin{equation*}
		f(\rho_h^{n+1}\circ r_i^{n+1}) - f(\rho_h^n \circ r_i^{n} ) \leq \tau f'(\rho_h^{n+1} \circ r_i^{n+1})  \bar{\partial} (\rho_h^n \circ r_i^n),
	\end{equation*}
where $\tau = t^{n+1}-t^n$.
	By using these formulas, we have
	
	\begin{equation}
		\begin{split}
			\mathcal{E}_h(\rho_h^{n+1}) - \mathcal{E}_h(\rho_h^{n})&= \sum_{i=1}^N \Big(\int_{I_i^{n+1}} f(\rho_h^{n+1}) dx - \int_{I_i^n}  f(\rho_h^{n}) dx \Big) \\
			& = \sum_{i=1}^N \Big(\int_{0}^{1} f(\rho_h^{n+1}\circ r_i^{n+1}) |I_i^{n+1}|ds - \int_{0}^{1} f(\rho_h^n \circ r_i^{n} ) |I_i^n|ds  \Big) \\
			& = \sum_{i=1}^N \Big( \int_0^1 (f(\rho_h^{n+1}\circ r_i^{n+1}) - f(\rho_h^n \circ r_i^{n} ))|I_i^{n+1}|ds  \\
			&\quad + \int_0^1 f(\rho_h^n \circ r_i^{n} ) (|I_i^{n+1}| -|I_i^n| ) \Big) ds \\
			& \leq  \tau\sum_{i=1}^N \Big( \int_0^1 f'(\rho_h^{n+1} \circ r_i^{n+1}) |I_i^{n+1}| (\rho_h^{n+1}\circ r_i^{n+1} - \rho_h^n \circ r_i^n) ds  \\
			&\quad + \int_0^1 f(\rho_h^n \circ r_i^n)  (\bar{\partial }x_{i}^n - \bar{\partial}x_{i-1}^n) ds \Big).    
			\label{e:temp0}     
		\end{split}
	\end{equation}

Notice the expression 
	\begin{equation}\label{e:temp1}
		\mathcal{E}_h(\rho_h^n) = \int_{I^n} f(\rho_h^n) dx =\sum_{i=1}^N \int_0^1 f(\rho_h^n\circ r_i^n) |I_k^n|ds.
	\end{equation}
	From the first equality of the above equation, we know that 
	$$\frac{\partial \mathcal
		{E}_h^{n}}{\partial x_i^{n}} = \int_{I^n} f'(\rho_h^n) \psi_i^n dx. $$
From the second equality of \eqref{e:temp1} and noticing that $\rho_h\circ r_i^n  = \rho_{i-1}^n + s(\rho_i^n - \rho_{i-1}^n)$, which does not depend on $x_i^n$, we obtain
	$$\frac{\partial \mathcal
		{E}_h^{n}}{\partial x_i^{n}} = \int_0^1 f(\rho_h^n \circ r_i^n)|_{I_i^n}ds - \int_0^1 f(\rho_h^n \circ r_{i+1}^n)|_{I_{i+1}^n} ds.$$
Thus we have the formula of $\frac{\partial \mathcal
		{E}_h^{n}}{\partial x_i^{n}}$:
	\begin{align*}
	\begin{split}
		\frac{\partial \mathcal
			{E}_h^{n}}{\partial x_i^{n}} &= \int_{I^n} f'(\rho_h^n) \psi_i^n dx = \int_0^1 f(\rho_h^n \circ r_i^n)|_{I_i^n}ds - \int_0^1 f(\rho_h^n \circ r_{i+1}^n)|_{I_{i+1}^n} ds.
	\end{split}
	\end{align*}
Similary, we have the formula of $\frac{\partial \mathcal
		{E}_h^{n+1}}{\partial \rho_i^{n+1}}$:
	\begin{align*}
		\begin{split}
			\frac{\partial \mathcal
				{E}_h^{n+1}}{\partial \rho_i^{n+1}} &= \int_{I^{n+1}} f'(\rho_h^{n+1}) \phi_i^{n+1} dx \\
			&= \int_0^1 f'(\rho_h^{n+1}\circ r_i^{n+1}) |I_i^{n+1}|sds + \int_0^1 f'(\rho_h^{n+1}\circ r_{i+1}^{n+1}) |I_{i+1}^{n+1}|(1-s)ds.
		\end{split}
	\end{align*}
	Then we  derive the expressions for $\sum_{i=1}^{N-1} \frac{\partial \mathcal
		{E}_h^{n+1}}{\partial \rho_i^{n+1}} \bar{\partial} \rho_i^n $ and $\sum_{i=0}^N 	\frac{\partial \mathcal
		{E}_h^{n}}{\partial x_i^{n}} \bar{\partial}x_i^n$ as 
	\begin{equation*}
		\begin{split}
			\sum_{i=1}^{N-1} \frac{\partial \mathcal
				{E}_h^{n+1}}{\partial \rho_i^{n+1}} \bar{\partial} \rho_i^n &=\sum_{i=1}^N  \int_0^1 f'(\rho_h^{n+1} \circ r_i^{n+1}) |I_i^{n+1}| ((1-s)\bar{\partial}\rho_{i-1}^{n} - s \bar{\partial} \rho_i^n) ds\\
			& =\sum_{i=1}^N  \int_0^1 f'(\rho_h^{n+1} \circ r_i^{n+1}) |I_i^{n+1}| (\rho_h^{n+1}\circ r_i^{n+1} - \rho_h^n \circ r_i^n) ds  ,
		\end{split}
	\end{equation*}
and
	\begin{equation*}
		\sum_{i=0}^N 	\frac{\partial \mathcal
			{E}_h^{n}}{\partial x_i^{n}} \bar{\partial}x_i^n=\sum_{i=1}^N\int_0^1 f(\rho_h^n \circ r_i^n)  (\bar{\partial }x_{i}^n - \bar{\partial}x_{i-1}^n) ds,
	\end{equation*}
respectively. By using these relations, the equation \eqref{e:temp0} is reduced to
	\begin{equation*}
		\mathcal{E}_h(\rho_h^{n+1}) - \mathcal{E}_h(\rho_h^{n}) \leq  \tau\sum_{i=1}^{N-1} \frac{\partial \mathcal
			{E}_h^{n+1}}{\partial \rho_i^{n+1}}\bar{\partial}\rho_i^n + \tau\sum_{i=0}^{N} \frac{\partial \mathcal{E}_h^n}{\partial x_i^n} \bar{\partial}x_i^n .  
	\end{equation*}
	By using the equations \eqref{scheme2_2}-\eqref{scheme2_4}, we can  further calculate 
	\begin{equation*}
		\begin{split}
			&\tau\sum_{i=1}^{N-1} \frac{\partial \mathcal
				{E}_h^{n+1}}{\partial \rho_i^{n+1}}\bar{\partial}\rho_i^n + \tau\sum_{i=0}^{N} \frac{\partial \mathcal{E}_h^n}{\partial x_i^n} \bar{\partial}x_i^n \\
			&= \tau \sum_{i=1}^{N-1} \Big(	\int_{I^n} \phi_i^n \lambda_h^{n+1} dx\Big) \bar{\partial} \rho_i^n \\
			&\quad + \tau \sum_{i=0}^N\Big(\int_{I^n} \psi_i^n \lambda_h^{n+1} dx -\int_{I^n} \rho_h^n \partial_x \lambda_h^{n+1} \phi_i^n dx -\int_{I^n} \rho_h^n v_h^{n+1} \phi_i^n dx \Big)\bar{\partial}x_i^n\\
			& = \tau \sum_{i=1}^{N-1}\Big(\int_{I^n} \Big(\sum_{j=1}^{N-1} \bar{\partial}{\rho}_j^n\phi_j^n \Big) \phi_i^n + \int_{I^n} \Big (\sum_{j=0}^N v_j^{n+1} \psi_j^n \Big)  \phi_i^n dx - \int_{I^n} \rho_h^n v_h^{n+1} \partial_x\phi_i^n dx \Big)\lambda_i^n\\
			&\quad - \tau \sum_{i=0}^{N} \Big( \int_{I^n} \rho_h^n v_h^{n+1} \phi_i^n dx \Big) \bar{\partial}x_i^n = - \tau  \int_{I^n} \rho_h^n (v_h^{n+1} )^2 dx \leq 0.
		\end{split}
	\end{equation*}
	Thus we obtain the following result:
	\begin{equation*}
		\mathcal{E}_h(\rho_h^{n+1}) - \mathcal{E}_h(\rho_h^{n}) \leq 0.
	\end{equation*}
\end{proof}

\subsection{Implementations of the numerical schemes}
The fully discrete explicit scheme \eqref{eq:ExEuler1}-\eqref{eq:ExEuler3} can be written as
\begin{align}
	\b{M}^n \b{\lambda}^{n+1} &=  \frac{\partial \mathcal{E}_h^n}{\partial \b{\rho}^n} , \label{matrix_scheme1_2}\\
	\b{D}^n \b{v}^{n+1}  &= - \frac{\partial \mathcal{E}_h^n}{\partial \b{x}^n} + (\b{B}^n - \b{E}^n)^T \b{\lambda}^{n+1}, \label{matrix_scheme1_3}\\
	\b{M}^n \bar{\partial}\b{\rho}^n &= -(\b{B}^n-\b{E}^n)\b{v}^{n+1},\label{matrix_scheme1_4}
\end{align}
where $\b{M}^n\in \mathbb{R}^{N-1,N-1}, \b{D}^n\in \mathbb{R}^{N+1,N+1}, \b{B}^n\in \mathbb{R}^{N-1,N+1}, \b{E}^n\in \mathbb{R}^{N-1,N+1},$ such that
\begin{align*}
	M_{ij}^n &= \int_{I^n} \phi_i^n \phi_j^n dx ;\quad	D_{ij}^n = \int_{I^n} \rho_h^n \phi_i^n \phi_j^n dx; \\
	B_{ij}^n &= \int_{I^n} \phi_i^n \psi_j^n dx ; \quad E_{ij}^n = \int_{I^n} \rho_h^n \partial_x \phi_i^n \phi_j^n dx .
\end{align*}
Since $\b{M}^n$ is the mass matrix and  $\b{D}^n$ is a modified mass matrix, both of them are positive definite if the value $\b{\rho}^n$ and $\b{x}^n$ satisfy Assumptions (A1) and (A2). In implementations, we can first solve \eqref{matrix_scheme1_2} to compute $\b{\lambda}^{n+1}$. Then we successively solve \eqref{matrix_scheme1_3} to obtain $\b{v}^{n+1}$ and \eqref{matrix_scheme1_4} to get $\b{\rho}^{n+1}$. Finally, we update $\b{x}^{n+1}$ by $\b{x}^{n+1} = \b{x}^n + \tau \b{v}^{n+1}$. The linear system \eqref{matrix_scheme1_2}-\eqref{matrix_scheme1_4} are decoupled and are easy to solve. 

The implicit scheme \eqref{scheme2_2}-\eqref{scheme2_4} can be written as 
\begin{align}
	\b{M}^n \b{\lambda}^{n+1} - \frac{\partial \mathcal{E}_h^{n+1}}{\partial \b{\rho}^{n+1}}  &=0 , \label{matrix_scheme2_2}\\
	\b{D}^n \b{v}^{n+1}   -(\b{B}^n - \b{E}^n)^T \b{\lambda}^{n+1}&= - \frac{\partial \mathcal{E}_h^n}{\partial \b{x}^n} , \label{matrix_scheme2_3}\\
	\b{M}^n \bar{\partial}\b{\rho}^n+(\b{B}^n-\b{E}^n)\b{v}^{n+1} &=0.\label{matrix_scheme2_4}
\end{align}
Since the coupled system \eqref{matrix_scheme2_2}-\eqref{matrix_scheme2_4} is nonlinear, we can choose a fixed-point iteration to solve them. The iteration termination condition in our implementations is $\|\b{\rho}^{n+1,k+1}-\b{\rho}^{n+1,k}\|_{l^{\infty}}\leq \epsilon$, $\|\b{x}^{n+1,k+1}-\b{x}^{n+1,k} \|_{l^{\infty}}\leq \epsilon $ with $\epsilon = 10^{-6}$. In each iterate step, we solve a linear problem. 

%

\begin{myRem}
	The derivation and the theoretical results in this section can be generalized to higher dimensional cases straightforwardly. More details on the methods  in the two dimensional case are given in the appendix. 
	Numerical examples in two dimensions are given in the Section 5.
\end{myRem}
\section{Modified numerical schemes}
Notice that  the conservation of the total mass of the semi-discrete scheme \eqref{eq:matrixEL} in the previous section is not guaranteed, as discussed in Remark~\ref{remark1}. This is due to the fact that we assume the homogeneous boundary condition of the Lagrange multiplier(or the pressure) on the free boundary, i.e. $\lambda_0=\lambda_N=0$. 
This condition is consistent with the continuous problem(see the equation \eqref{eq:EL}). 
However, we need consider an alternative boundary condition to preserve the mass conservation  for the semi-discrete scheme.

We still use a piecewisely linear approximation for $\rho$ and allow the grid nodes move in a Lagrange manner.
Namely, we let $\rho_h=\sum_{i=1}^{N-1}\rho_i\phi_i$ and $v_i=\dot{x}_i$ for $i=0, \cdots, N$.  In comparison with the previous derivation, we do not propose the boundary condition for the Lagrangian multiplier and let $\widehat{\lambda}_h =\sum_{i=0}^{N}\lambda_i \phi_i$. 
Then  the augmented Rayleighean functional is defined as 
\begin{equation}
\widehat{\mathcal R}_h= \Phi_h +\dot{\mathcal{E}}_h - \int_I(\partial_t \rho_h + \partial_x(\rho_h v_h) )\widehat \lambda_h dx.
\end{equation}
This corresponding variational problem is
\begin{equation}
	\begin{aligned}
		&\min_{\dot{\b{\rho}},\b{v}} 	\mathcal{R}_h(\b{\rho},\b{x}; \dot{\b{\rho}},\b{v}) \\
		 &s.t. \quad \int_{I} (\partial_t \rho_h + \partial_x(\rho_h v_h) ) w_h dx = 0, \quad \forall w_h \in V_{h}^t.
	\end{aligned}		\label{eq:ConstraintP1}
\end{equation}
The only difference from the problem \eqref{eq:ConstraintP} is that the test function  $w_h$ belongs to $V_h^t$ instead of its subspace $V_{h,0}^t$. 
By  similar derivations in the previous section, we can derive the following Euler-Lagrange equation
\begin{equation}
	\begin{cases}
 \frac{\partial \mathcal
		{E}_h}{\partial \rho_i}  -  \int_I \phi_i\lambda_hdx=0,& i=1,...,N-1 ;\\
	\int_I \rho_h  v_h \phi_i dx + \frac{\partial \mathcal{E}_h}{\partial x_i} - \int_I \psi_i\lambda_h dx +\int_I \rho_h \partial_x\lambda_h\phi_i dx =0,& i=0,...,N;\\
	 \int_I \partial_t \rho_h \phi_i dx - \int_I \rho_h v_h \partial_x\phi_i dx=0 ,& i=0,...,N.
	 \end{cases}
 \label{eq:discreteEL1}
\end{equation}
In an algebraic form, the equation is rewritten as
\begin{equation}
	\begin{cases}
	\widehat{\b{M}}(\b{x}(t))\widehat{\b{\lambda}}(t) = \frac{\partial \mathcal
		{E}_h}{\partial \b{\rho}}(\b{x}(t),\b{\rho}(t)),\\
	\b{D}(\b{x}(t),\b{\rho}(t))\dot{\b{x}}(t) =- \frac{\partial \mathcal
		{E}_h}{\partial \b{x}}(\b{x}(t),\b{\rho}(t)) + \Big(\widehat{\b{B}}(\b{x}(t)) - \widehat {\b{E}}(\b{x}(t),\b{\rho}(t)) \Big)^{T} \widehat{\b{\lambda}}(t) ,\\
\widehat{\b{M}}^T(\b{x}(t))\dot{\b{\rho}}(t) + \Big(\widehat {\b{B}}(\b{x}(t)) - \widehat {\b{E}}(\b{x}(t),\b{\rho}(t)) \Big)\dot{\b{x}}(t)=0,
	\end{cases}
\label{eq:matrixEL1}
\end{equation}
where $\b{D}$ is the same as in \eqref{eq:matrixEL}, $\widehat{\b{M}}\in \mathbb{R}^{N-1,N+1}, \widehat{\b{B}}\in \mathbb{R}^{N+1,N+1}, \widehat{\b{E}}\in \mathbb{R}^{N+1,N+1}$ such that
\begin{align*}
&\widehat{{M}}_{ij}(\b{x}(t)) = \int_{I(t)} \phi_i\phi_j dx ; 
\quad
\widehat	{{B}}_{ij}(\b{x}(t)) = \int_{I(t)} \phi_i \psi_j dx ;\\
&\widehat{{E}}_{ij}(\b{x}(t),\b{\rho}(t)) = \int_{I(t)} \rho_h \partial_x \phi_i \phi_j dx.
\end{align*}

For the semi-discrete problem~\eqref{eq:discreteEL1}, we have the following proposition,
\begin{proposition}\label{prop:MassConserv1}
	Under the assumptions (A1) and (A2),  let $\b{\rho}(t),\b{x}(t)$ be the solution of the equations \eqref{scheme1_1} and \eqref{eq:discreteEL1}, then  the total mass is conserved in the sense that, 
	\begin{equation}
		\frac{d}{dt} \int_{I(t)} \rho_h(x,t)  dx =0.
		\label{eq:mass_cons1}
	\end{equation}
\end{proposition}
\begin{proof}
Similar to the analysis in the proof in Proposition~\ref{prop:MassConserv0}, we have
	\begin{equation*}
		\frac{d}{dt}\int_{I(t)} \rho_h(x,t) \phi_i(x,t) dx = 0,\quad i=0,...,N.
	\end{equation*}
	Noticing that $\sum_{i=0}^{N} \phi_i=1$,
the equation~\eqref{eq:mass_cons1} is obtained by adding all the above equations together.
\end{proof}

We can also develop an explicit linear scheme and an implicit nonlinear scheme by time discretization. However, the explicit scheme is not decoupled in this case. The two schemes in algebraic forms are respectively given by,
\begin{align}
	& \widehat{\b{M}}^n \widehat{\b{\lambda}}^{n+1} =  \frac{\partial \mathcal{E}_h^n}{\partial \b{\rho}^n} , \label{matrix_scheme3_2}\\
	& \b{D}^n \b{v}^{n+1}  - (\widehat{\b{B}}^n - \widehat{\b{E}}^n)^{T} \b{\lambda}^{n+1} = - \frac{\partial \mathcal{E}_h^n}{\partial \b{x}^n}, \label{matrix_scheme3_3}\\
& (\widehat{\b{M}}^n)^T \bar{\partial}\b{\rho}^n +(\widehat{\b{B}}^n-\widehat{\b{E}}^n)\b{v}^{n+1} = 0;\label{matrix_scheme3_5}
\end{align}
and 
\begin{align}
	& \widehat{\b{M}}^n \widehat{\b{\lambda}}^{n+1} - \frac{\partial \mathcal{E}_h^{n+1}}{\partial \b{\rho}^{n+1}}  =0 , \label{matrix_scheme4_2}\\
	&\b{D}^n \b{v}^{n+1}   - (\widehat{\b{B}}^n - \widehat{\b{E}}^n)^T \widehat{\b{\lambda}}^{n+1}= - \frac{\partial \mathcal{E}_h^n}{\partial \b{x}^n}, \label{matrix_scheme4_3}\\
&(\widehat{\b{M}}^n)^T \bar{\partial}\b{\rho}^n+ (\widehat{\b{B}}^n - \widehat{\b{E}}^n)\b{v}^{n+1} =0.\label{matrix_scheme4_5}
\end{align}

%
%
Although the semi-discrete scheme~\eqref{eq:discreteEL1} satisfies the mass conservation property, the fully discrete schemes do not satisfy the property since they are linearized schemes. Our numerical experiments show that the numerical scheme \eqref{matrix_scheme3_2}-\eqref{matrix_scheme3_5}(or \eqref{matrix_scheme4_2}-\eqref{matrix_scheme4_5})) gives almost the same results as those in the previous section. Therefore, we will use the simpler schemes \eqref{matrix_scheme1_2}-\eqref{matrix_scheme1_4} and \eqref{matrix_scheme2_2}-\eqref{matrix_scheme2_4} in the  numerical examples  next section.

\section{Numerical examples}
In this section we present some numerical results to  show the effectivity of our numerical methods. We  consider both one dimensional and two dimensional problems.

We choose the Barenblatt-Pattle solution \cite{zel1950towards,barenblatt1952some} to test the accuracy of our method. The Barenblatt-Pattle solution is a special solution for the PME in $\mathbb{R}^d$. Let $\mathbf{x}$ be the coordinate of a point in $\mathbb{R}^d$. This solution has an explicit form
\begin{equation}
	B(\mathbf{x},t) = t^{-\alpha} (C - k|\mathbf{x}|^2 t^{-2\beta})_{+}^{\frac{1}{m-1}},
\end{equation}
where $(s)_{+} = \max \{s,0\}$,
\begin{equation*}
	\alpha = \frac{d}{d(m-1)+2},\quad \beta = \frac{\alpha}{d}, \quad k = \frac{\alpha (m-1)}{2md},
\end{equation*}
and $d$ is the number of dimension, $C>0$ is a constant  determined by the total mass. This solution has a compact support in space for any fixed time $t$. The free boundary is the surface given by the equation
\begin{equation}
	t = c|\mathbf{x}|^{d(m-1)+2},
\end{equation}
where $c= (\frac{k}{C})^{\frac{d(m-1)+2}{2}}$. The boundary changes position when $t$ increases.

The $L^2$ error between the discrete solution $\rho_h$ and the exact solution $\rho$ at time $T$ is computed by
\begin{equation}
	err_{L^2} := \Big(\int_{\Omega} (\rho(\mathbf{x},T)-\rho_h(\mathbf{x},T))^2 d\mathbf{x} \Big)^{1/2}.
\end{equation}
\subsection{One-dimensional problems}
 Numerical experiments show that both the explicit and implicit schemes work well when the time step is small. In the one-dimensional case, we employ the implicit scheme \eqref{scheme2_2}-\eqref{scheme2_4} to solve the PME.  We are mainly interested in the adaptive motion of the mesh and how it affects the numerical errors.
\subsubsection{Convergence tests}
	We first consider the Barenblatt-Pattle solution, with constants $C = 1$ and $d = 1$. We take the Barenblatt-Pattle solution at $t=1$, denoted as $B(x,1)$, as the initial data. We compare the numerical solution with the exact solution $B(x,T)$ at time $T=2$, for both $m=2$ and $m=5$ cases. To investigate the accuracy of the method in the one-dimensional case, we use a uniform mesh and a least squares best fit mesh  \cite{baines1994algorithms} for the initial data, respectively. In both cases, the time-step is reduced by a factor of 4 each time the mesh is refined.
	
	Figure~\ref{fig:1ab} shows the convergence behavior of our method in various situations. 
	From Figure \ref{fig1}, we observe that when $m=2$, the second-order convergence rate can be obtained on a uniform initial mesh. However, when $m=5$, the uniform initial mesh leads to a slower convergence rate. This is due to the fact that larger $m$ corresponds to a more singular solution of the PME. To cure the discrepancy of the convergence rate, we can choose a better initial
	mesh by finding a best approximation to the initial function in a piecewisely linear finite element space with free knots. We do this by
	a least square method \cite{baines1994algorithms}. The numerical results for the nonuniform initial meshes are shown in Figure~\ref{fig2}. We see that the optimal
	convergence rate is obtained for both   $m=2$ and $m=5$ cases. 
	

	\begin{figure}[ht!]
		\centering
		\subfigure[uniform initial mesh]{
			\begin{minipage}[ht]{0.5\linewidth}
				\centering
				\includegraphics[width=1.05\textwidth]{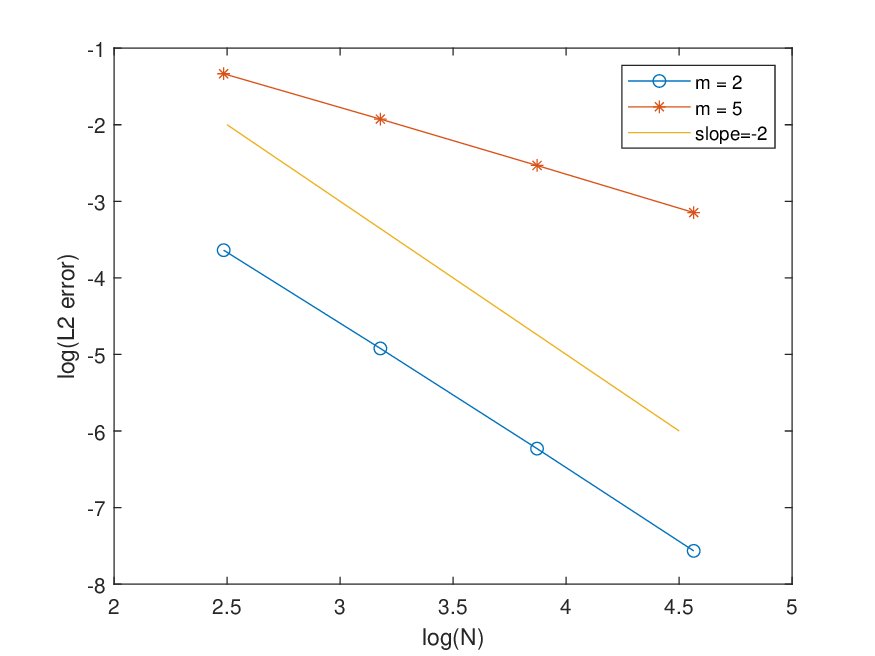}
				\label{fig1}
			\end{minipage}%
		}%
		\subfigure[non-uniform initial mesh]{
			\begin{minipage}[ht]{0.5\linewidth}
				\centering
				\includegraphics[width= 1.05\textwidth]{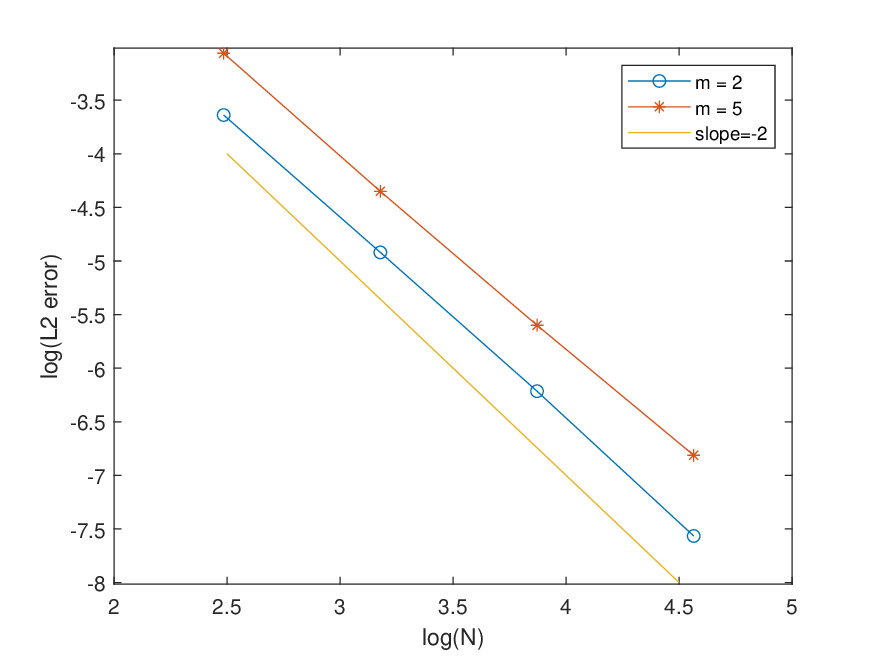}
				\label{fig2}
			\end{minipage}
		}%
		\caption{Convergence of the numerical solutions for the PME ($m = 2, 5$) at $T = 2$.}	\label{fig:1ab}
	\end{figure}  

	Figure \ref{fig3} exhibits the numerical and exact  (Barenblatt-Pattle) solutions at $T=2$ for $m=5$ using different meshes. 
	We see that the numerical solutions fit well with the exact solution even for a very coarse mesh. This implies that the boundary
	points moves correctly when time evolves.
	\begin{figure}[ht!]
		\centering
		\subfigure[N=12]{
			\begin{minipage}[h]{0.45\linewidth}
				\centering
				\includegraphics[width=0.9\textwidth]{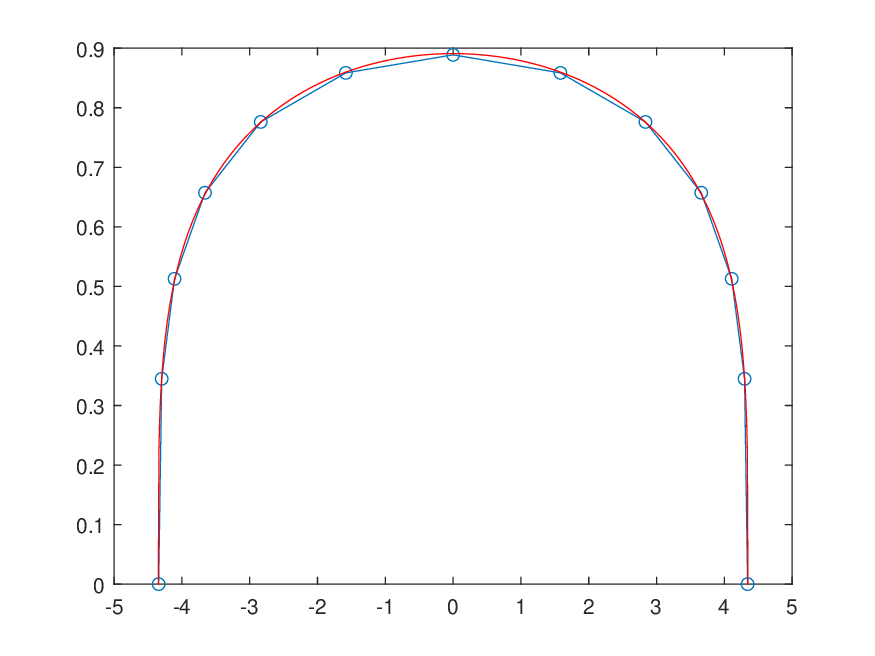}
			\end{minipage}%
		}%
		\subfigure[N=24]{
			\begin{minipage}[h]{0.45\linewidth}
				\centering
				\includegraphics[width= 0.9\textwidth]{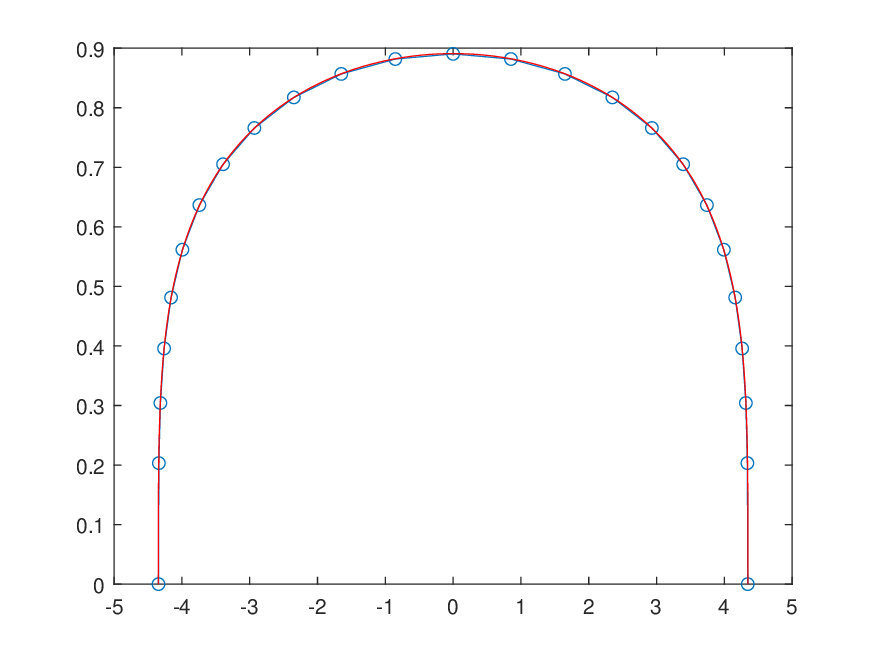}
			\end{minipage}
		}%
		\qquad
		\subfigure[N=48]{
			\begin{minipage}[h]{0.45\linewidth}
				\centering
				\includegraphics[width=0.9\textwidth]{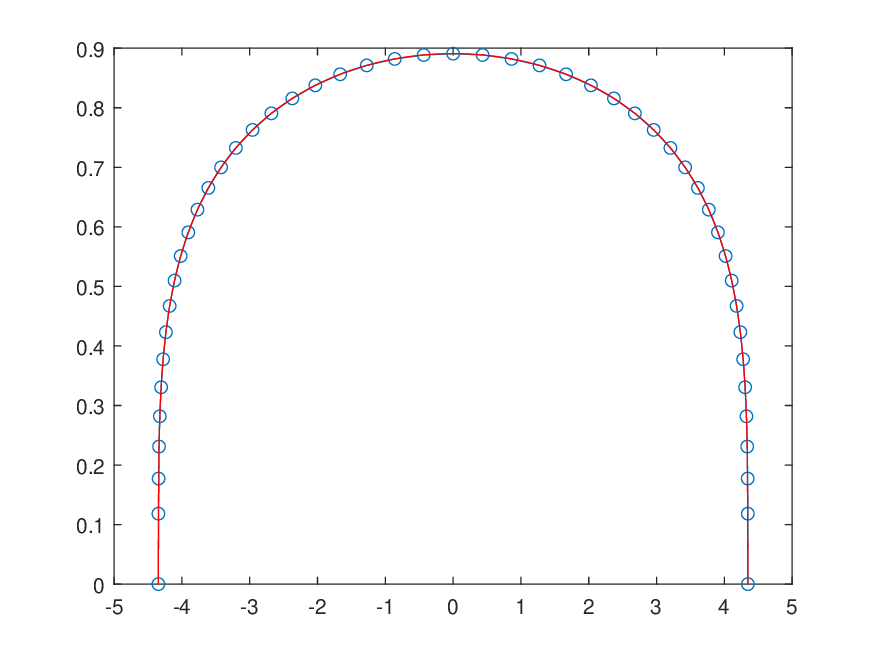}
			\end{minipage}%
		}%
		\subfigure[N=96]{
			\begin{minipage}[h]{0.45\linewidth}
				\centering
				\includegraphics[width= 0.9\textwidth]{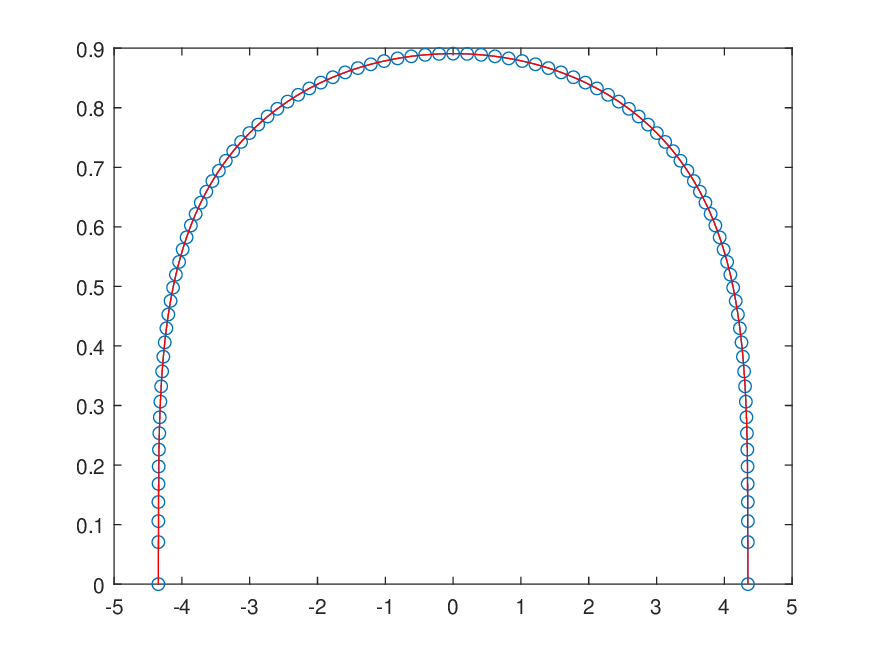}
			\end{minipage}
		}
		\caption{The exact solutions (red solid line) and the numerical solutions (blue circles) of the PME($m=5$).}
		\label{fig3}
	\end{figure}
	
	As stated in Section 3, the numerical scheme \eqref{eq:discreteEL} does not preserve the total mass exactly. 
	Nevertheless,  numerical experiments show that the  errors for the mass are usually  small. Some typical results 
	are listed in Table~\ref{table:Mass}. 
	In the test,	we choose $C = \frac{3^{1/3}}{4}$ so that the mass $\int_{I} B(x,1)dx = 1$ for $m=2$. We see that the errors for the total mass are very small and decay  with an optimal convergence
	rate.

	\begin{table*}
		\begin{center}
			\caption{Convergence of the mass errors at $T=2$}
			\label{table:Mass}
			\begin{tabular}{c|c|c}
				\hline
				Number of mesh points & Error of numerical mass & Order \\
				\hline
				12  & $5.5471\times 10^{-4}$ &       \\
				24 &            $1.3884 \times 10^{-4}$                               & 1.9983       \\
				48 &       $3.4719 \times 10^{-5}$                                    &      1.9996 \\
				96 &            $8.6804 \times 10^{-6}$                & 1.9999 \\
				\hline
			\end{tabular}
		\end{center}
	\end{table*}

\subsubsection{Waiting time phenomenon}
It is known that  the solution of the PME may exhibit a waiting time phenomenon.  Namely  the support of the solution may not change until the time $t$ is larger than some critical value $t^*$. To test if our method can capture this phenomenon, we consider an
 initial value as in \cite{aronson1983initially},
	\begin{equation}
		\rho_0(x) = \begin{cases}
			\Big(\frac{m-1}{m}((1-\theta)\sin^2(x) + \theta \sin ^4(x))\Big)^{\frac{1}{m-1}},\quad x\in [-\pi,0],\\
			0, \quad \text{ otherwise},
		\end{cases}
		\label{waiting_time}
	\end{equation}
	with $\theta \in [0,1]$. For such an initial value,  the critical waiting time is given by an explicit formula  $
	t^*=\frac{1}{2(m+1)(1-\theta)}$ when $\theta \in [0,\frac{1}{4}]$.
	
	In our tests, we set $\theta = 0$ and $m = 4$. This gives a waiting time $t^*=0.1$. The numerical solutions at various time are shown in Figure~\ref{fig4}. We see that the shape of the solution changes while the support of the solution does not change until $t\geq 0.1$.
To show the motion of the boundary points more clearly, we plot the coordinates of the left and right boundaries with respect time in
Figure~\ref{fig:waiting time point}. We can see that the numerical method can capture the waiting time phenomenon automatically.

	\begin{figure}[ht!]
		\centering
		\subfigure[t=0]{
			\begin{minipage}[t]{0.45\linewidth}
				\includegraphics[width=0.95\textwidth]{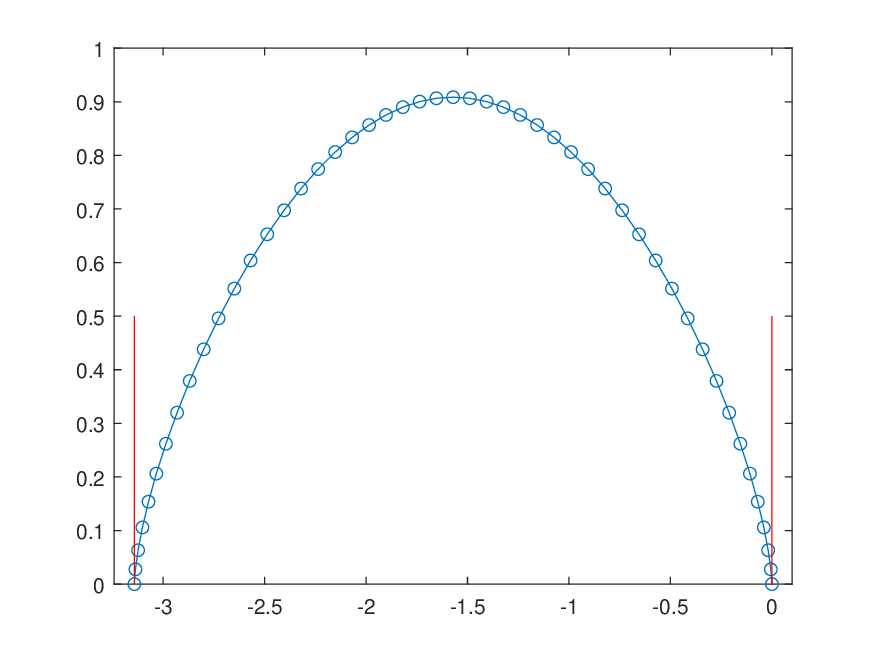}
			\end{minipage}%
		}%
		\subfigure[t=0.1]{
			\begin{minipage}[t]{0.45\linewidth}
				\includegraphics[width=0.95\textwidth]{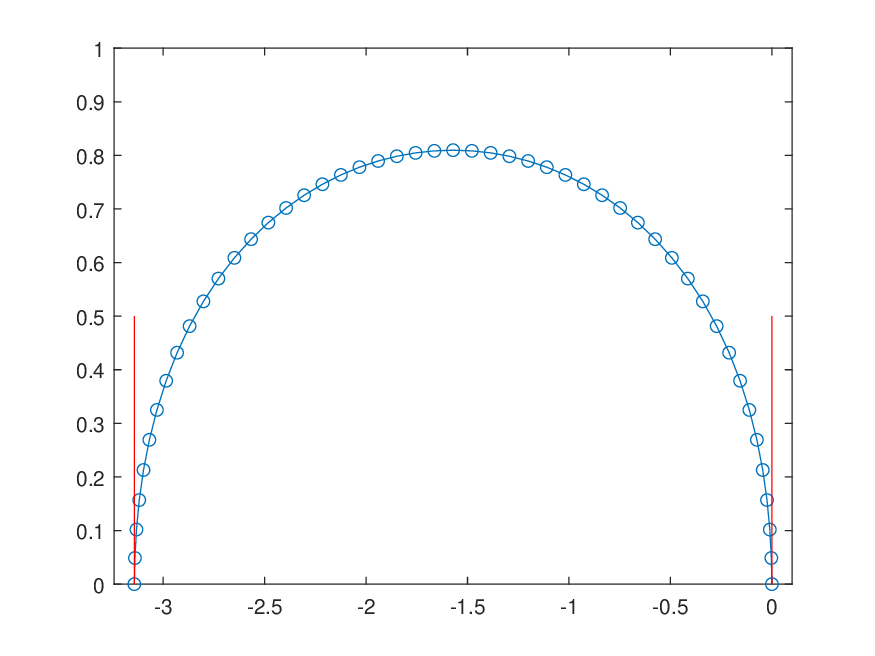}
			\end{minipage}%
		}%
		\\
		\subfigure[t=0.15]{
			\begin{minipage}[ht]{0.45\linewidth}
				\includegraphics[width= 0.95\textwidth]{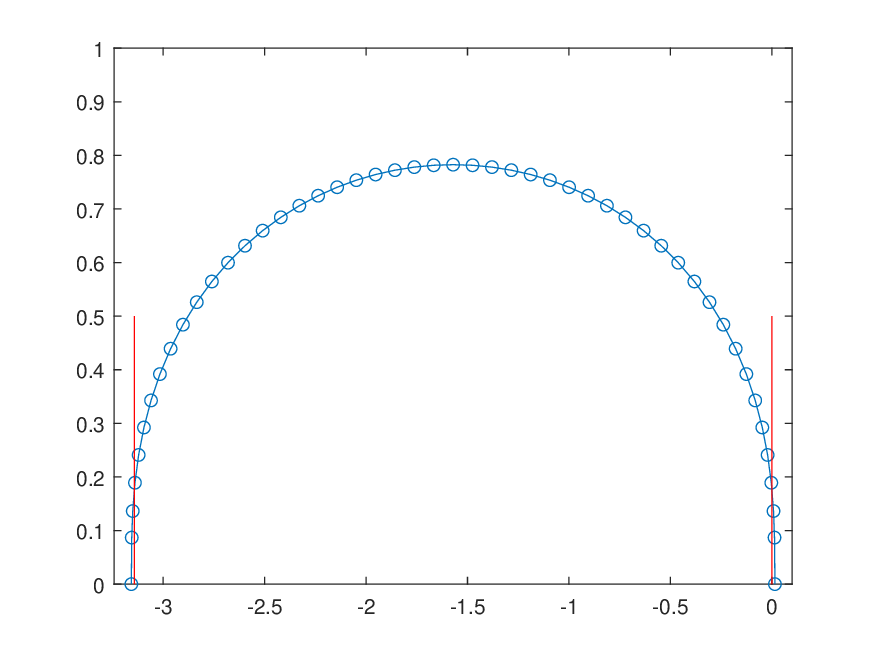}
			\end{minipage}
		}
		\subfigure[t=0.2]{
			\begin{minipage}[ht]{0.45\linewidth}
				\includegraphics[width= 0.95\textwidth]{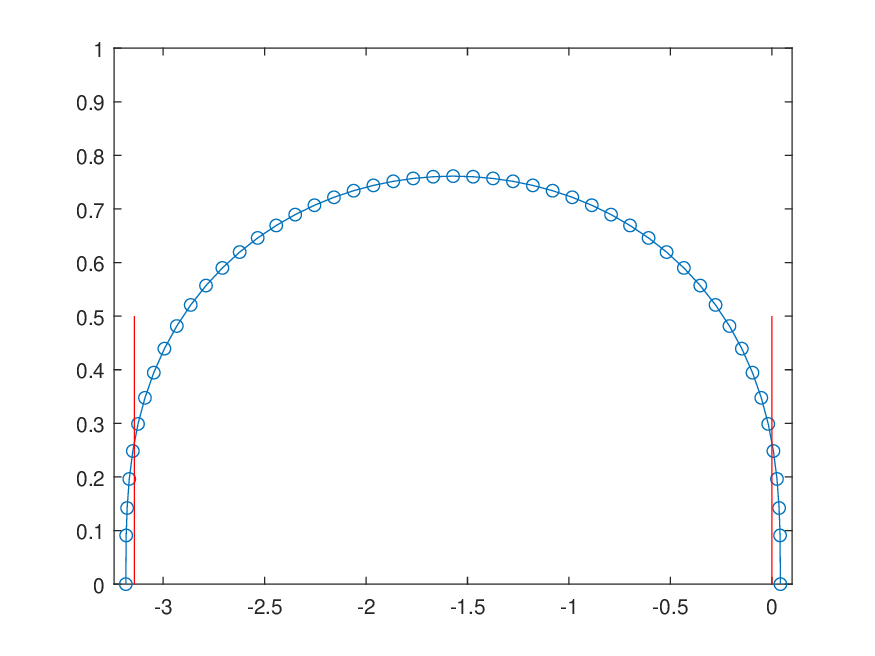}
			\end{minipage}
		}
		\caption{Numerical solutions of the PME($m=4$) with the initial value \eqref{waiting_time} with $\theta = 0$. Here $N = 48,  \tau = 2.5\times 10^{-4}$.}
		\label{fig4}
	\end{figure}

	\begin{figure}[ht!]
		\centering
		\subfigure[Left boundary]{
			\begin{minipage}[t]{0.5\linewidth}
				\centering
				\includegraphics[width=1.05\textwidth]{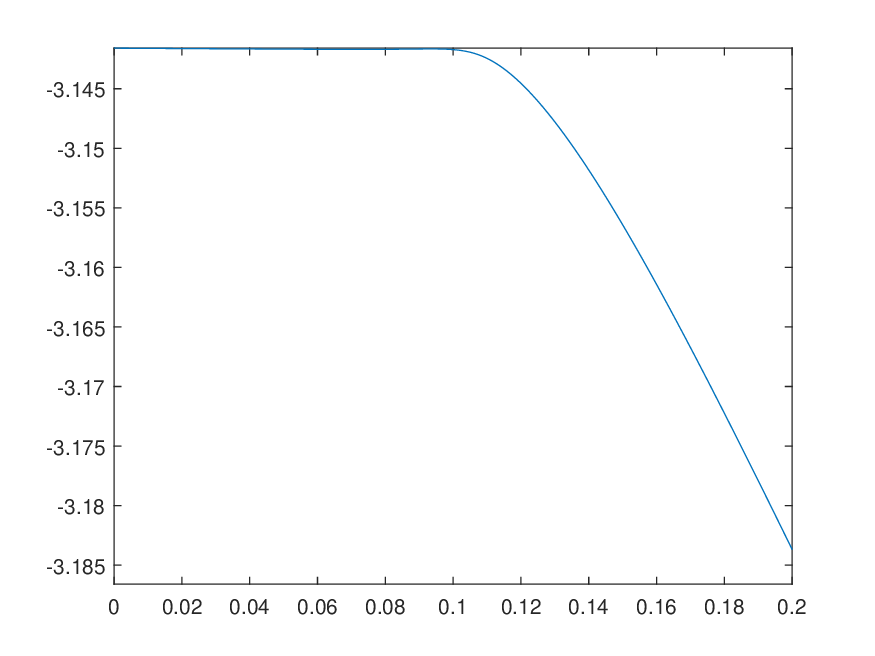}
			\end{minipage}%
		}%
		\subfigure[Right boundary]{
			\begin{minipage}[t]{0.5\linewidth}
				\centering
				\includegraphics[width= 1.05\textwidth]{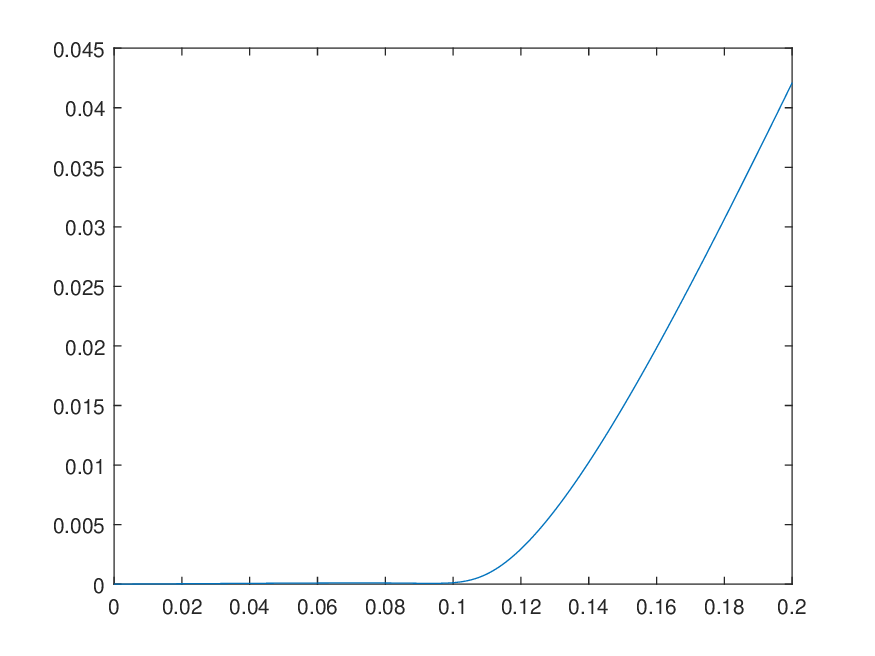}
			\end{minipage}
		}%
		\caption{Boundary motion with respect to time.}
		\label{fig:waiting time point}
	\end{figure}

\subsection{Two-dimensional problems}
We apply our numerical method to the PME in two dimensions. The explicit formulae of the numerical 
schemes are given in the appendix. In the numerical tests, we use the explicit numerical scheme for simplicity.
In 2D, the convergence rate $p$ is calculated by 
\begin{equation}
p=	\frac{\log(err_1/err_2)}{\log(\sqrt{N_2/N_1})},
\end{equation}
where $err_1$ and $err_2$ are  the $L^2$ errors  for the numerical solutions calculated on meshes with  $N_1$ and $N_2$ vertexes, respectively.

\subsubsection{Convergence test}
	We consider the  two dimensional Barenblatt-Pattle solution $B(x,y,t)$ with $C = 0.1$ and $d = 2$, where $(x,y)$ is 
	the coordinate of a point in $\mathbb{R}^2$.  We set $B(x,y,1)$ as the initial data and test the convergence rate at $T=2$. The numerical results are shown in Figure~\ref{fig:mesherror}. 
	The numerical results are similar to that in one dimensional case. We see that the optimal convergence rate is obtained when 
 $m=2$ with a quasi-uniform initial mesh. For larger $m$, the uniform initial mesh will lead to a sub-optimal convergence rate. This is shown in Figure~\ref{fig:mesherror} for the $m=5$ case.  We also see that that the convergence rate is better on a  non-uniform initial mesh than that on a uniform one. We remark that it is not an easy task to find an optimal initial mesh in 2D as discussed in \cite{baines1994algorithms}.

\begin{figure}
	\centering
	\includegraphics[width=0.75\textwidth]{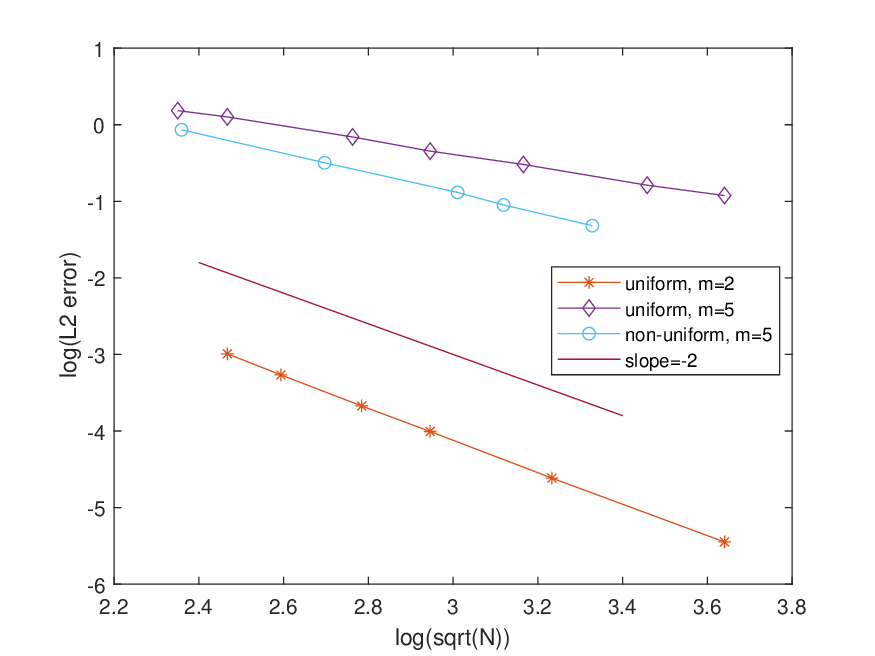}
	\caption{Convergence of the numerical solutions of the PME ($m = 2,5$)  at $T = 2$.}
	\label{fig:mesherror}
\end{figure}
 

We can also compute the waiting time phenomenon in the two dimensional case. For that purpose, we choose $m=2$ and
 the following initial function,
\begin{equation}
	\rho_0(x,y) = \begin{cases}
		\frac{1}{2}\sin^2(\sqrt{x^2+y^2}-\pi),\quad \hbox{if } \sqrt{x^2+y^2}\leq \pi, \\
		0,\qquad \qquad\qquad\qquad \qquad \hbox{otherwise}.
	\end{cases}
	\label{eq:waiting time 2D}
\end{equation}
According to the previous theoretical results \cite{vazquez2007porous}, there exists a positive waiting time for
such an initial value.  In the numerical test, the initial triangulation for $\Omega:=\{(x,y): \sqrt{x^2+y^2}<\pi\}$
is quasi-uniform with $1983$ cells. We set $\tau=10^{-3}$. 

Figure \ref{fig:waiting time} shows the numerical solutions of the PME for the initial data \eqref{eq:waiting time 2D} at various time. We  see that the numerical method can also  capture the waiting time phenomenon very well in the two dimensional case and
the waiting time is about $0.125$.
\begin{figure}[ht!]
	\centering
	\subfigure[t=0]{
		\begin{minipage}[t]{0.3\linewidth}
			\centering
			\includegraphics[width=0.95\textwidth]{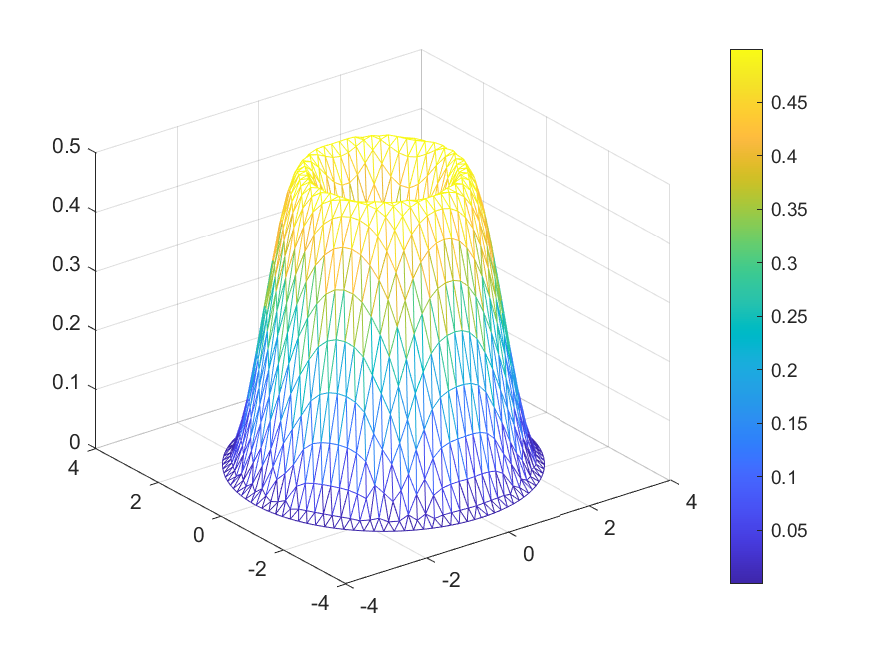}
		\end{minipage}%
	}%
	\subfigure[t=0]{
		\begin{minipage}[t]{0.3\linewidth}
			\centering
			\includegraphics[width= 0.95\textwidth]{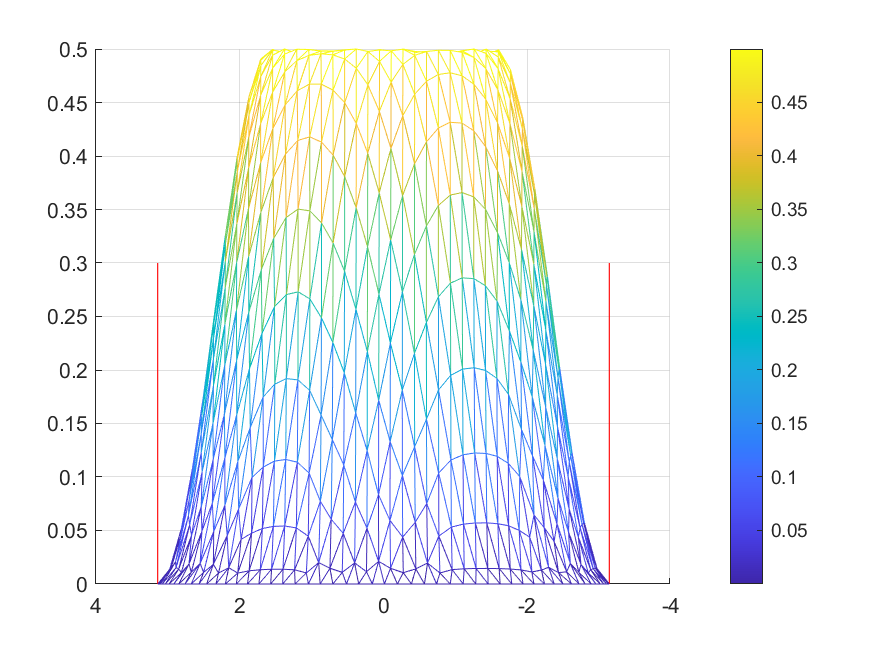}
		\end{minipage}
	}%
	\subfigure[t=0]{
		\begin{minipage}[t]{0.3\linewidth}
			\centering
			\includegraphics[width= 0.95\textwidth]{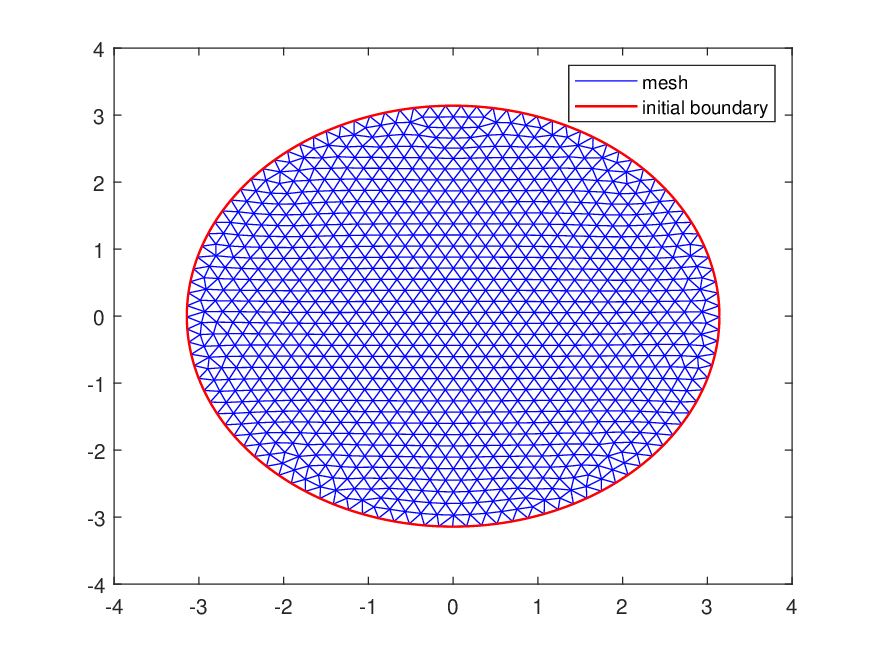}
		\end{minipage}
	}%
	\qquad
	\subfigure[t=0.125]{
		\begin{minipage}[t]{0.3\linewidth}
			\centering
			\includegraphics[width= 0.95\textwidth]{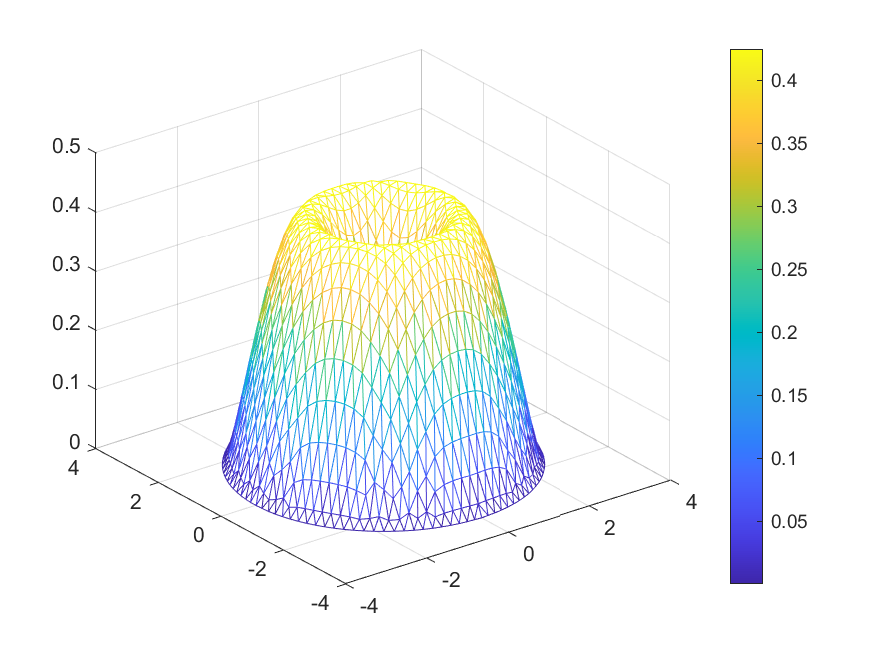}
		\end{minipage}
	}
	\subfigure[t=0.125]{
		\begin{minipage}[t]{0.3\linewidth}
			\centering
			\includegraphics[width= 0.95\textwidth]{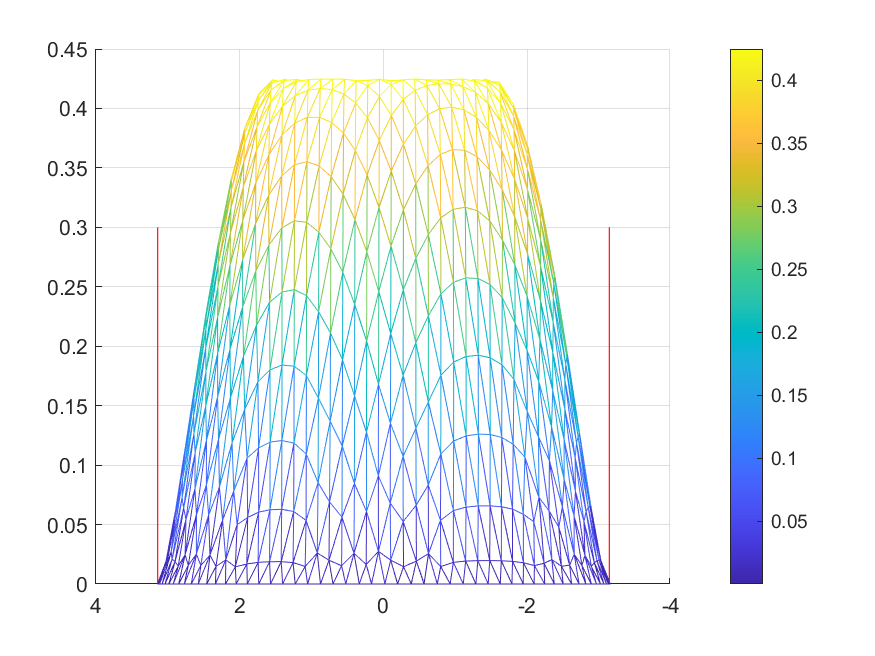}
		\end{minipage}
	}
	\subfigure[t=0.125]{
		\begin{minipage}[t]{0.3\linewidth}
			\centering
			\includegraphics[width= 0.95\textwidth]{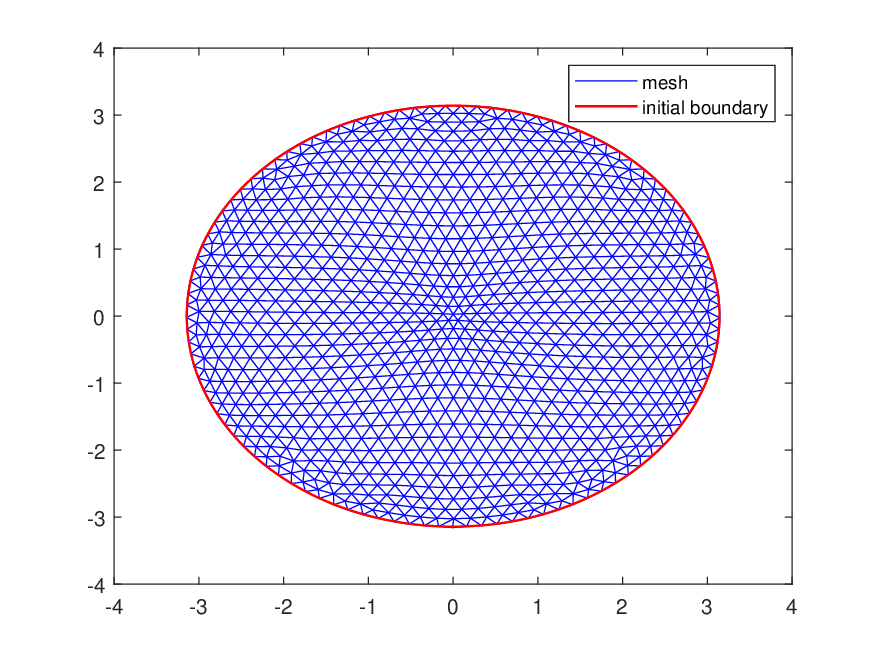}
		\end{minipage}
	}
	\subfigure[t=0.250]{
		\begin{minipage}[t]{0.3\linewidth}
			\centering
			\includegraphics[width= 0.95\textwidth]{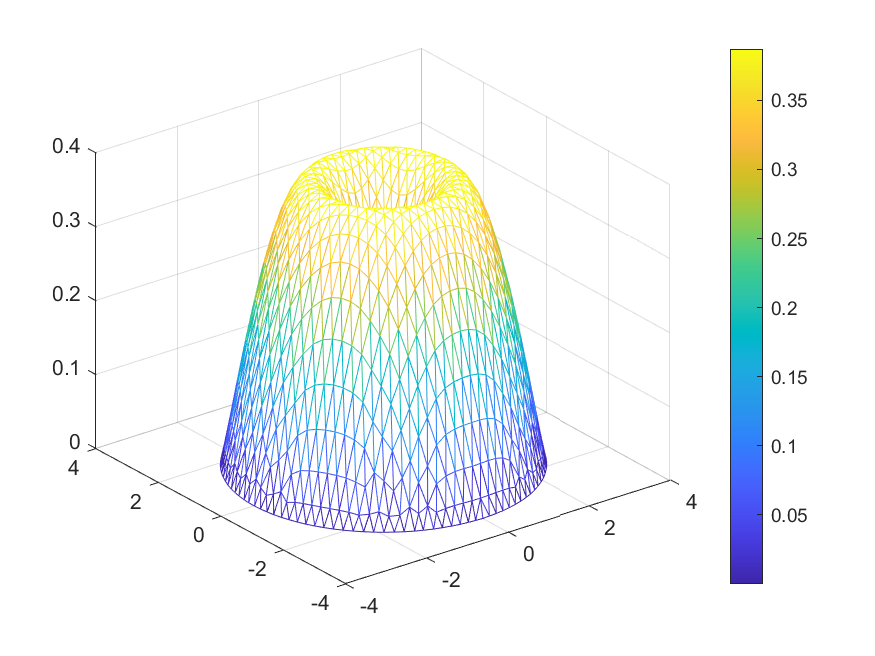}
		\end{minipage}
	}	\subfigure[t=0.250]{
		\begin{minipage}[t]{0.3\linewidth}
			\centering
			\includegraphics[width= 0.95\textwidth]{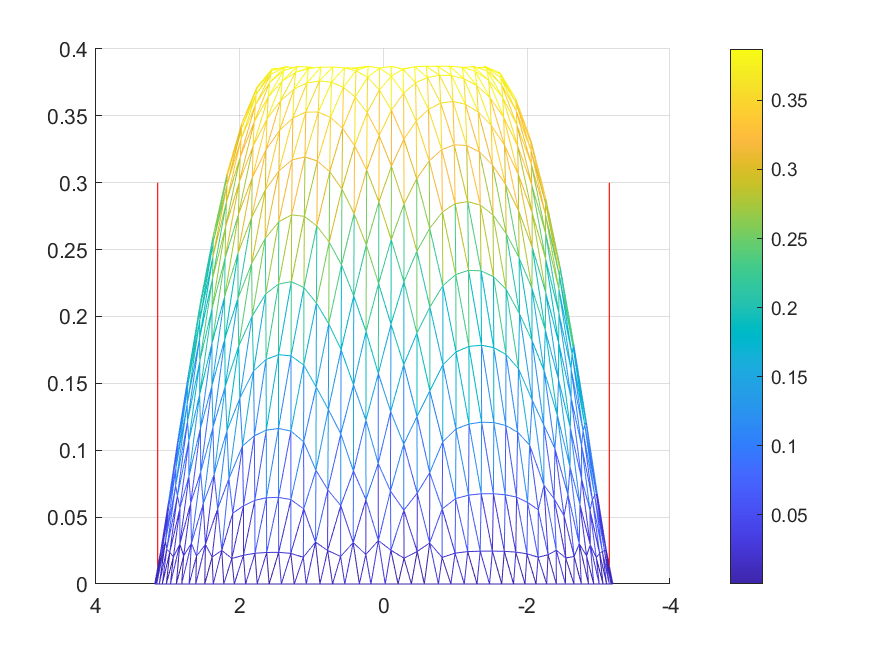}
		\end{minipage}
	}
	\subfigure[t=0.250]{
		\begin{minipage}[t]{0.3\linewidth}
			\centering
			\includegraphics[width= 0.95\textwidth]{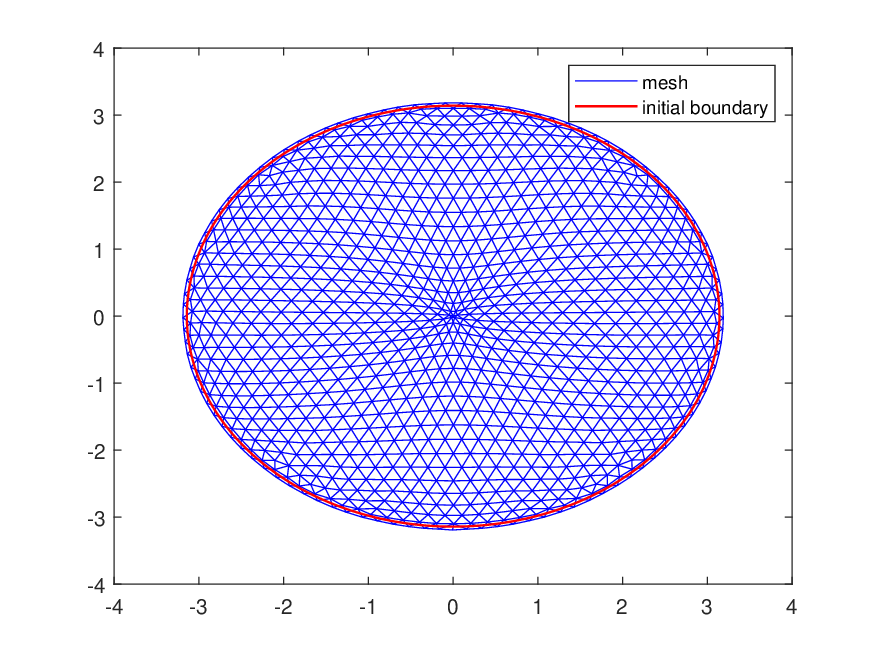}
		\end{minipage}
	}
	\caption{Numerical solutions of the PME($m=2$) with the initial value \eqref{eq:waiting time 2D}.}
	\label{fig:waiting time}
\end{figure}

\subsubsection{General examples}
Finally, we show some examples with more general initial values.
\begin{figure}[ht!]
	\centering
	\subfigure[t=0]{
		\begin{minipage}[t]{0.5\linewidth}
			\centering
			\includegraphics[width=0.9\textwidth]{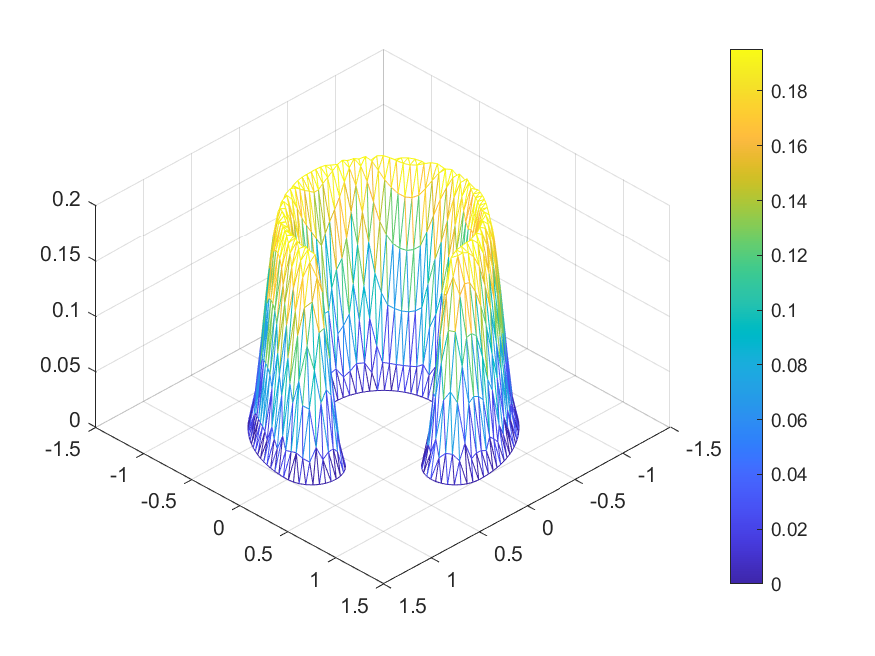}
		\end{minipage}%
	}%
	\subfigure[t=0]{
		\begin{minipage}[t]{0.5\linewidth}
			\centering
			\includegraphics[width= 0.9\textwidth]{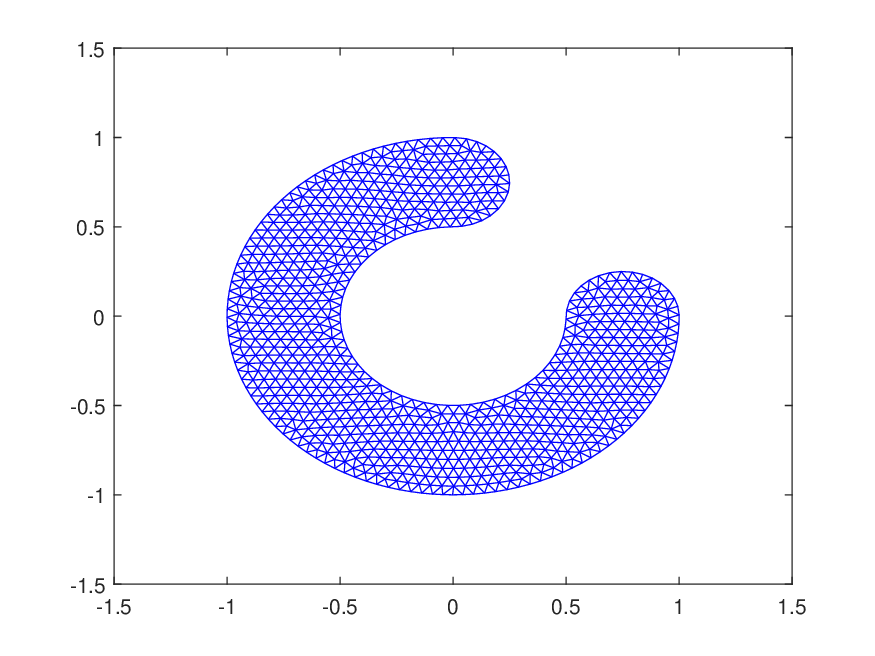}
		\end{minipage}
	}%
	\qquad
	\subfigure[t=0.1]{
		\begin{minipage}[t]{0.5\linewidth}
			\centering
			\includegraphics[width=0.9\textwidth]{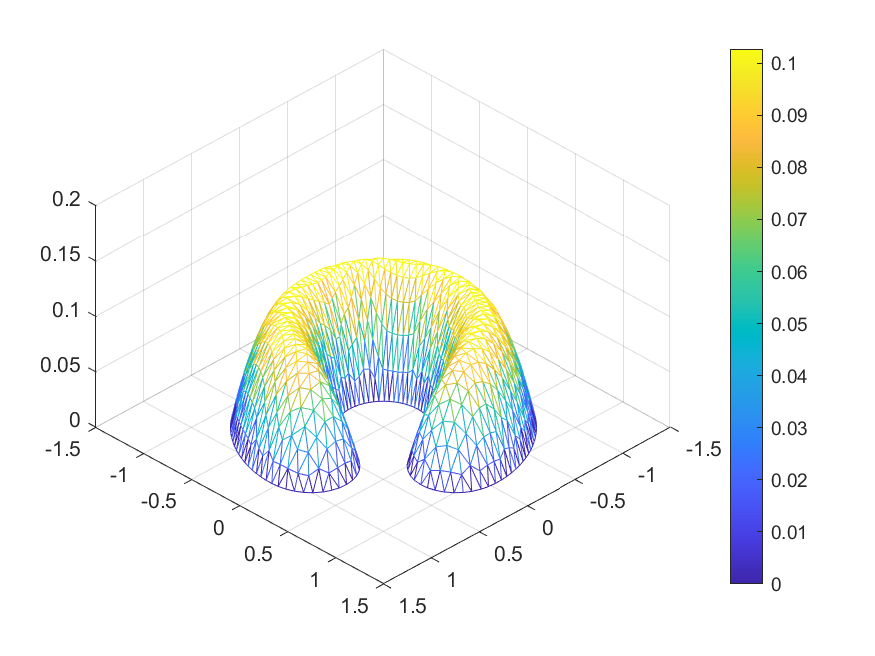}
		\end{minipage}%
	}%
	\subfigure[t=0.1]{
		\begin{minipage}[t]{0.5\linewidth}
			\centering
			\includegraphics[width= 0.9\textwidth]{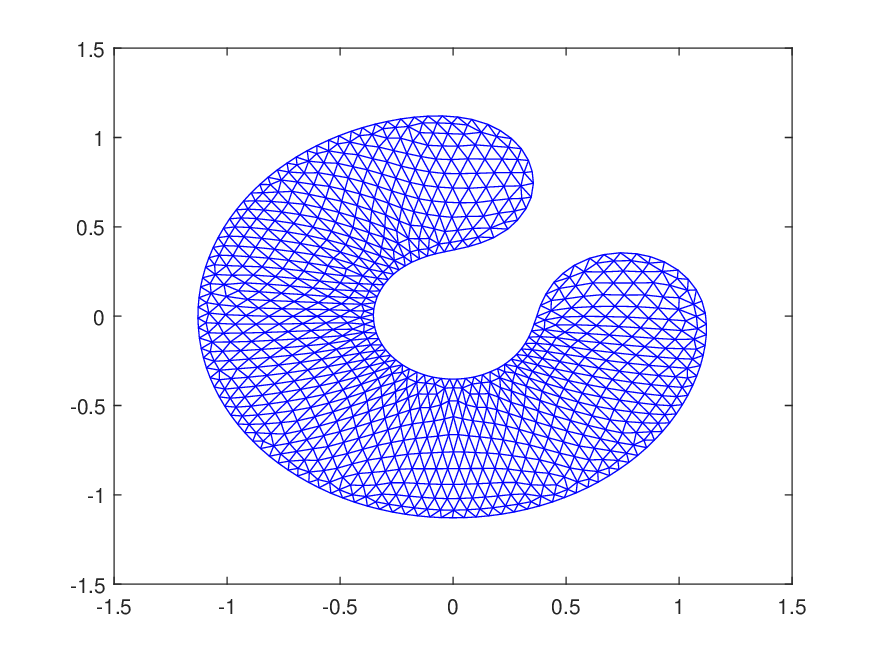}
		\end{minipage}
	}
	\qquad
	\subfigure[t=0.2]{
		\begin{minipage}[t]{0.5\linewidth}
			\centering
			\includegraphics[width=0.9\textwidth]{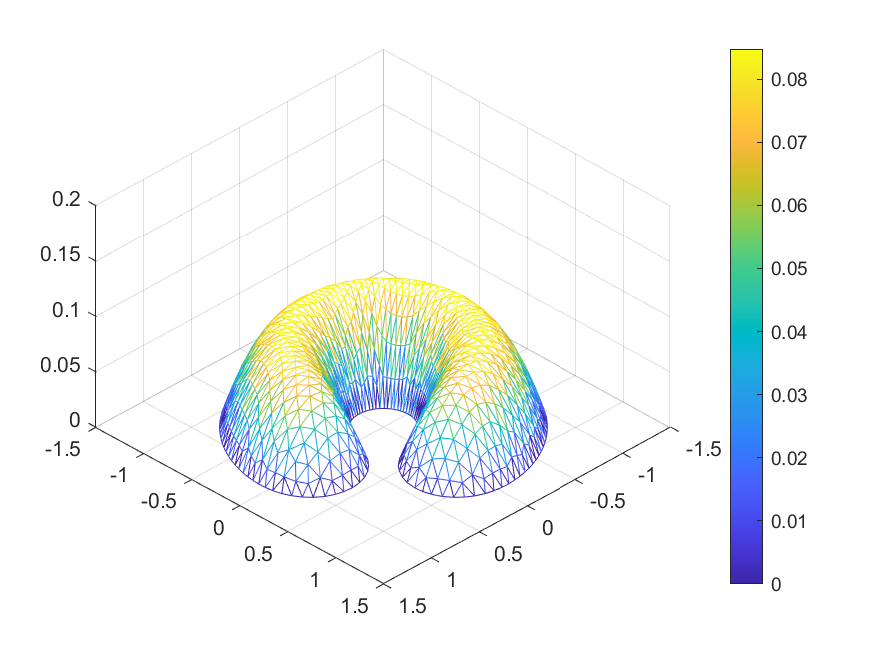}
		\end{minipage}%
	}%
	\subfigure[t=0.2]{
		\begin{minipage}[t]{0.5\linewidth}
			\centering
			\includegraphics[width= 0.9\textwidth]{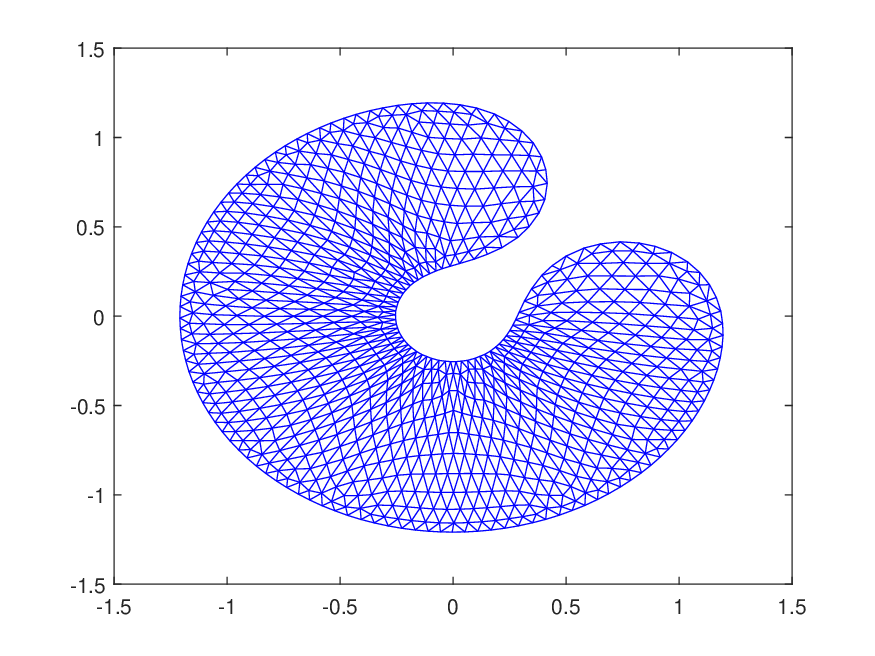}
		\end{minipage}
	}
	\caption{Numerical solutions of the  PME($m = 2$) with initial value  \eqref{eq:complex}. The initial mesh is  quasi-uniform and has  $910$ nodes, $1684$ cells. The time step is $\tau = 10^{-3}$.}
	\label{fig:complex}
\end{figure}
{
We first consider  an initial value with a compact support of ``horseshoe" shape, as  in \cite{baines2005moving,ngo2019adaptive,liu2020lagrangian}. In the test, we set $m=2$ and the initial function  is given by 
	\begin{equation}
		\rho_0(x,y) = \begin{cases}
			50(0.25^2-(\sqrt{x^2+y^2}-0.75)^2)^{2},\\
			\qquad\qquad\qquad \quad \hbox{if } \sqrt{x^2+y^2}\in (0.5,1) \text{ and } (x<0 \text{ or } y<0);\\
			50(0.25^2-x^2-(y-0.75)^2)^2,\\
			\qquad\qquad\qquad \quad \hbox{if } x^2+(y-0.75)^2\leq 0.25^2 \text{ and } x\geq 0;\\
			50(0.25^2-(x-0.75)^2-y^2)^2,\\
			\qquad\qquad\qquad \quad \hbox{if } (x-0.75)^2+y^2\leq 0.25^2 \text{ and } y\leq 0;\\
			0,	\qquad\qquad\qquad  \text{otherwise}.
		\end{cases}
		\label{eq:complex}
	\end{equation}
	Figure \ref{fig:complex} illustrates how the solution evolves with time. We see our method can solve the problem very well until the boundary of the support intersecting each other. However, the present numerical method cannot directly deal with the topology  change. 
	
	To deal with the topology change, a possible way is to consider a regularized problem where the PME is extended to a larger region and the initial value is set to be a small positive constant in the outer region.
	In the following, we show such an example with a solution with two peaks merging into one for the PME with $m=3$, motivated by the work in \cite{carrillo2018lagrangian,liu2020lagrangian}. Let $\Omega = [-1.5,1.5]^2$ and the initial data is given by
	\begin{equation}
		\rho_0(x,y) = e^{-20*((x-0.3)^2+(y-0.3)^2)} + e^{-20*((x+0.3)^2+(y+0.3)^2)} + 0.001,
		\label{eq:complex2}
	\end{equation} 
	which has two peak regions connected by a very thin layer with thickness $0.001$. The numerical results are shown in Figure \ref{fig:complex2}. It is clearly seen that the two separate peaks merges together gradually.
	\begin{figure}[ht!]
		\centering
		\subfigure[t=0.02]{
			\begin{minipage}[t]{0.5\linewidth}
				\centering
				\includegraphics[width=0.9\textwidth]{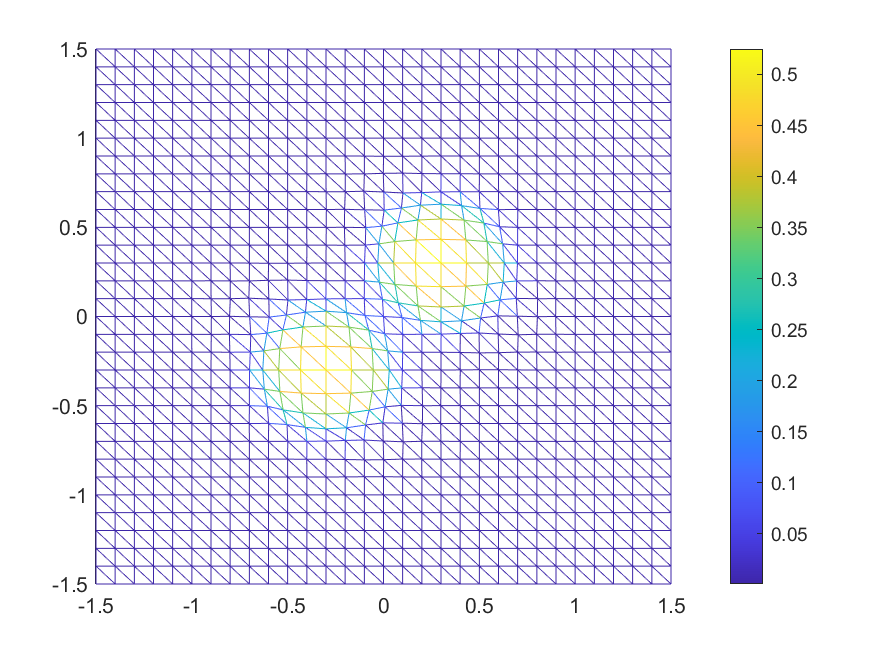}
			\end{minipage}%
		}%
		\subfigure[t=0.05]{
			\begin{minipage}[t]{0.5\linewidth}
				\centering
				\includegraphics[width= 0.9\textwidth]{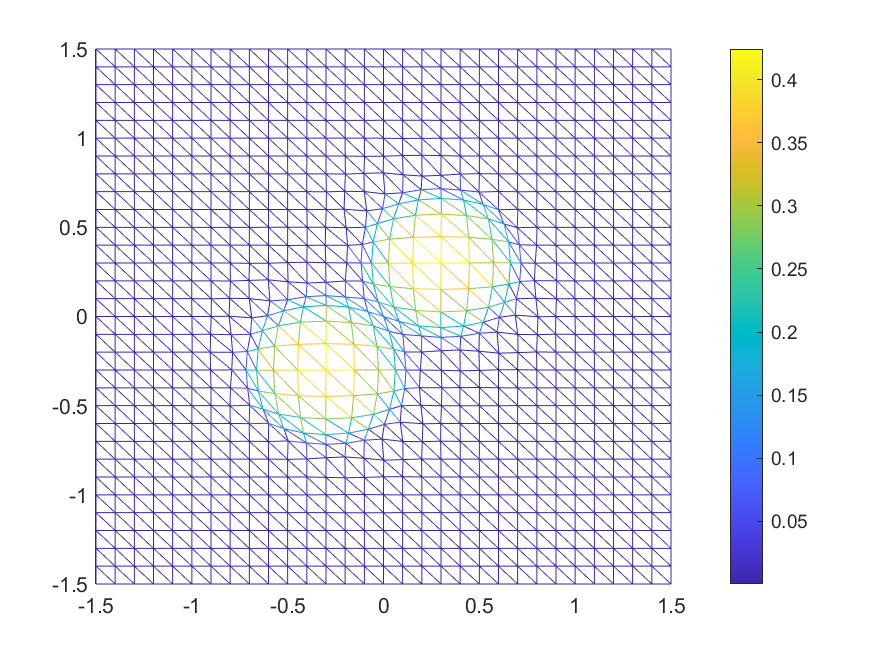}
			\end{minipage}
		}%
		\qquad
		\subfigure[t=0.1]{
			\begin{minipage}[t]{0.5\linewidth}
				\centering
				\includegraphics[width=0.9\textwidth]{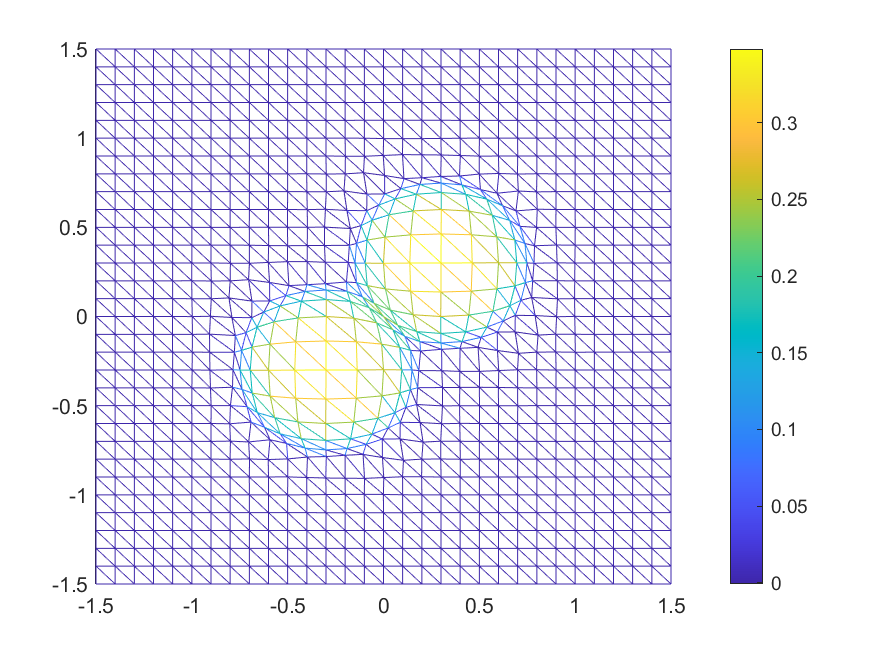}
			\end{minipage}%
		}%
		\subfigure[t=0.2]{
			\begin{minipage}[t]{0.5\linewidth}
				\centering
				\includegraphics[width= 0.9\textwidth]{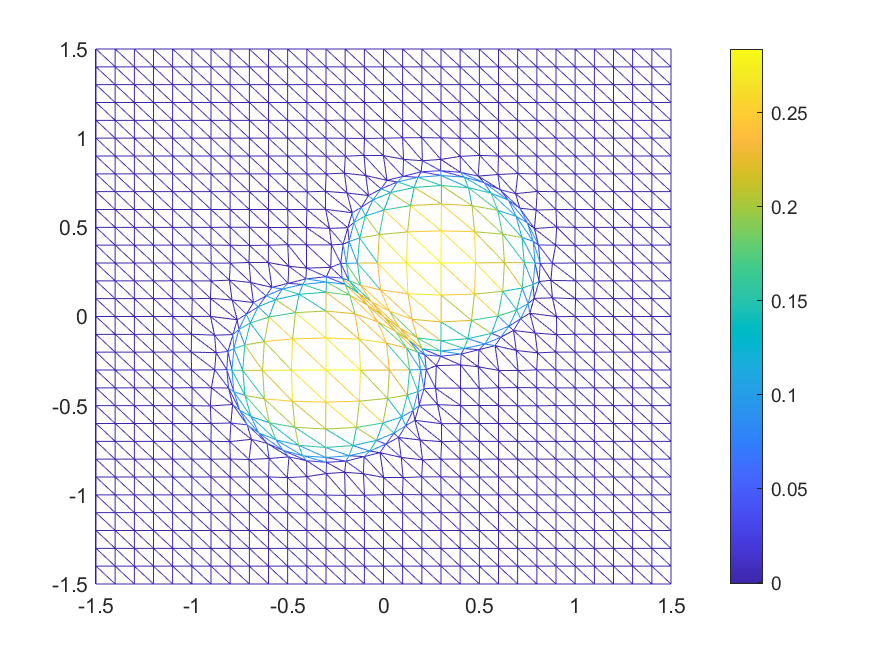}
			\end{minipage}
		}
		\qquad
		\subfigure[t=0.3]{
			\begin{minipage}[t]{0.5\linewidth}
				\centering
				\includegraphics[width=0.9\textwidth]{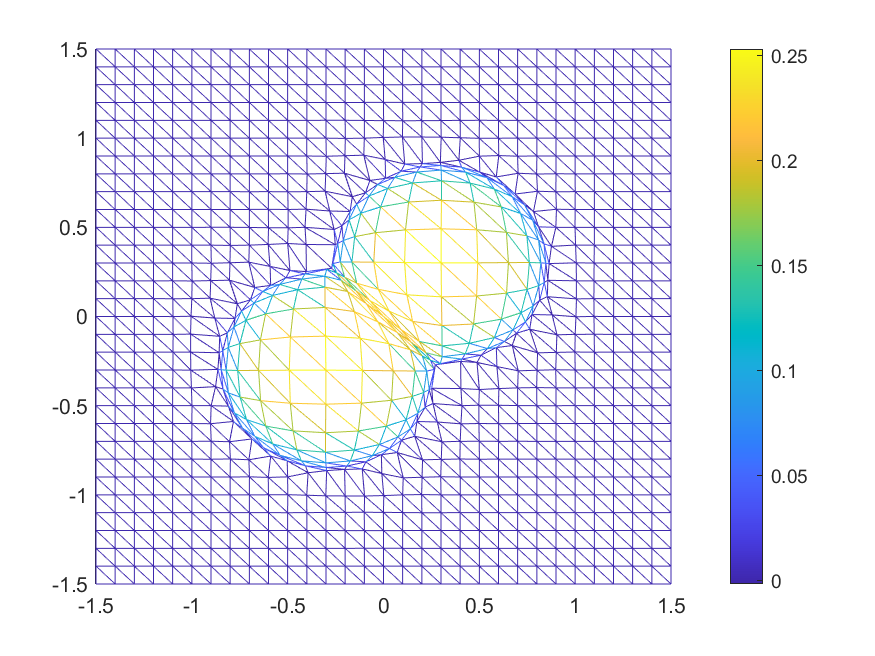}
			\end{minipage}%
		}%
		\subfigure[t=0.4]{
			\begin{minipage}[t]{0.5\linewidth}
				\centering
				\includegraphics[width= 0.9\textwidth]{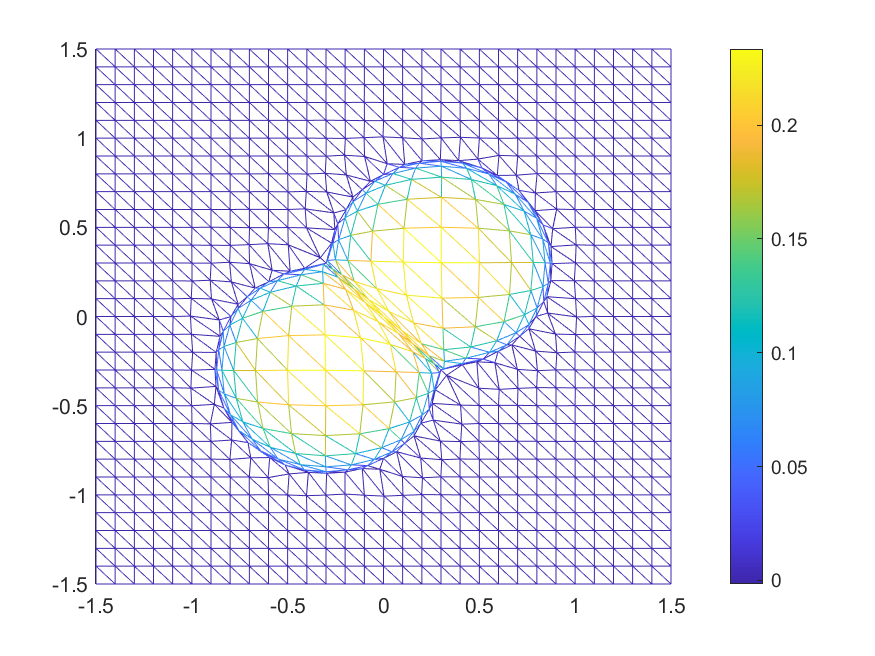}
			\end{minipage}
		}
		\caption{The numerical solutions for the  PME ($m = 3$) with initial value  \eqref{eq:complex2}. The initial mesh is uniform with 961 nodes, 1800 cells. The time step is $\tau = 10^{-3}$.}
		\label{fig:complex2}
	\end{figure}
	
}

\section{Conclusion}

In this paper, we utilize the Onsager principle to develop a new moving mesh method for the porous medium equation. We demonstrate that both the continuous PME and a semi-discrete scheme can be derived using this principle, ensuring that the scheme maintains the same energy dissipation structure as the continuous problem. Additionally, we introduce a fully discrete explicit decoupled scheme and an implicit scheme. Numerical examples illustrate the effectiveness of both schemes, showing that optimal convergence rates for the $L^2$
error can be achieved when the initial meshes are appropriately selected. The method naturally captures the waiting time phenomena and can be extended to higher-dimensional problems and higher-order approximations. It is important to note that while the derivation of the method is intuitive, the error estimate for the moving mesh method remains an open question. Optimal convergence estimates will require suitable assumptions regarding the initial meshes and the regularity properties of the system.

\appendix\renewcommand{\appendixname}{Appendix}
\section{Numerical scheme in two dimensions}
In this appendix, we will give the discrete numerical scheme in the two dimensional case.
 Let $\Omega(t) \subset \mathbb{R}^2$ be the domain where the PME is defined. 
 Denote by  $\mathcal{T}_h$  a partition of $\Omega$ with $N_K$ disjoint triangles $K$ such that the union of the triangles compose a polygonal domain $\Omega_h$. We suppose that the boundary vertices of $\Omega_h$ locate on $\partial \Omega$ initially. Let  $V_h^t$ be the finite element space with respect to the mesh $\mathcal{T}_h$,
\begin{equation}
V_h^t : = \{ u_h \in C(\bar{\Omega}_h)| u_h \text{ is linear in } K_i,i=1,...,N_K \}.
\end{equation}
 Let $N_{in}$ be the number of  vertices inside $\Omega_h$, and let $N$ be the  total number of the  vertexes in $\bar{\Omega}_h$. 
 Denote by $P_j=(x_j,y_j)$, ${j=1},\cdots, N$, a vertex of $\mathcal{T}_h$ which may change position with time $t$. Denote by $$V_{h,0}^t =\{u_h \in V_h^t:  u_h(P_j)=0, \forall P_j \text{  on } \partial\Omega_h \}.$$  
 Then the approximation $\rho_h(x,y,t) = \sum_{i=1}^{N_{in}} \rho_k(t)\phi_i(x,y,t)$, where $\phi_k(x,y,t), i=1,...,N_{in}$ are the global piecewisely linear finite element basis functions.
Similar to the one-dimensional case, the time derivative and space derivative of $\rho_h(x,y,t)$ are given by
\begin{align*}
\partial_t \rho_h &= \sum_{i=1}^{N_{in}} \dot{\rho}_i(t)\phi_i(x,y,t) + \sum_{i=1}^{N} (\dot{x}_i(t)\psi_{x,i}(x,y,t) + \dot{y}_i(t)\psi_{y,i}(x,y,t)),\\
\partial_x \rho_h &= \sum_{i=1}^{N_{in}} \rho_i(t)\partial_x \phi_i(x,y,t), 
		 	\end{align*}
		 			\begin{align*}
\partial_y \rho_h &= \sum_{i=1}^{N_{in}} \rho_i(t)\partial_y\phi_i(x,y,t),	
\end{align*}
where 
\begin{equation*}
\psi_{x,i} = \frac{\partial \rho_h}{\partial x_i}, \quad \psi_{y,i} = \frac{\partial \rho_h}{\partial y_i}.
\end{equation*}
Denote by $\b{\rho} = (\rho_1(t),...,\rho_{N_{in}}(t))^{T}, \b{x} = (x_1(t),...,x_N(t))^T$ and $\b{y}  =
 (y_1(t),...,y_N(t))^T$,
then the discrete energy functional $\mathcal{E}$ and its time derivative  are respectively given by
\begin{equation}
	\mathcal{E}_{h}(\b{\rho},\b{x},\b{y}) = \int_{\Omega_h} f(\rho_h)dxdy
\end{equation}
and 
\begin{equation}
\dot{\mathcal{E}}_h(\b{\rho},\b{x},\b{y};\dot{\b{\rho}},\dot{\b{x}},\dot{\b{y}}) =\sum_{i=1}^{N_{in}} \frac{\partial \mathcal{E}_h}{\partial \rho_i}\dot{\rho}_i +\sum_{i=1}^N \Big(\frac{\partial \mathcal{E}_h} {\partial x_i}\dot{x}_i + \frac{\partial \mathcal{E}_h} {\partial y_i}\dot{y}_i \Big),
\end{equation}	
where
\begin{align*}
&\frac{\partial \mathcal{E}_h}{\partial \rho_i} = \int_{\Omega_h} f'(\rho_h) \phi_i dxdy,\quad i=1,...,N_{in};\\
&	\frac{\partial \mathcal{E}_h} {\partial x_i} = \int_{\Omega_h} f'(\rho_h)\psi_{x,i} dxdy,\quad i=1,...,N;\\
&  \frac{\partial \mathcal{E}_h} {\partial y_i} = \int_{\Omega_h} f'(\rho_h)\psi_{y,i} dxdy,\quad i=1,...,N.
\end{align*}

Let $\b{v}_h(x,y,t) = (v_{x}^{h}(x,y,t),v_{y}^{h}(x,y,t))$ be an approximation  of velocity $\b{v}$, s.t.  $v_{x}^{h}(x,y,t) =\sum_{i=1}^N v_{x,i}(t) \phi_i(x,y,t)$ and $ v_{y}^{h}(x,y,t) =\sum_{i=1}^N v_{y,i}(t) \phi_i(x,y,t))$.  Let $\lambda_h(x,y,t)=\sum_{i=1}^{N_{in}} \lambda_i(t)\phi_i(x,y,t)$ be an approximation of $\lambda(x,y,t)$. Then we can obtain a discrete version of the dissipation function,
\begin{equation*}
	\Phi_h = \frac{1}{2} \int_{\Omega_h}\rho_h |\b{v}_h|^2 dx dy  =  \frac{1}{2} \int_{\Omega_h} \rho_h (v_{x,h}^2 + v_{y,h}^2) dx dy.
\end{equation*}
For the continuum equation, we have
\begin{align*}
\begin{split}
	&\sum_{K\in \mathcal{T}_h}\int_{K} [\partial_t \rho_h + \nabla \cdot (\rho_h \b{v}_h)] w_h dxdy = \sum_{K\in \mathcal{T}_h}\int_{K} (\partial_t \rho_h w_h -\rho_h \b{v}_h \cdot \nabla w_h )dxdy \\
	&=  \int_{\Omega_h}(\partial_t \rho_h w_h - \rho_h (v_{x,h} \partial_x w_h + v_{y,h} \partial_y w_h) )dxdy = 0, \quad \forall w_h \in V_{h,0}^t.
\end{split}
\end{align*}

Then we have the discrete Rayleighian functional with a Lagrange multiplier $\lambda_h$, 
\begin{equation}
\begin{split}
	\tilde{\mathcal{R}}_h &= \Phi_h + \dot{\mathcal{E}}_h - \int_{\Omega_h}\partial_t \rho_h \lambda_h - \rho_h (v_{x,h} \partial_x \lambda_h + v_{y,h} \partial_y \lambda_h) dxdy.
\end{split}
\label{DRay2}
\end{equation}
We directly compute the corresponding Euler-Lagrange equations,
\begin{align}
\frac{\partial \tilde{\mathcal{R}}_h}{\partial \dot{\rho}_i} &=  \frac{\partial \mathcal
	{E}_h}{\partial \rho_i}  -  \int_{\Omega_h} \phi_i \lambda_h dxdy=0,\qquad \qquad\ i=1,...,N_{in}; \\
\frac{\partial \tilde{\mathcal{R}}_h}{\partial v_{x,i}} &=
\int_{\Omega_h} \rho_h v_{x,h} \phi_i dxdy + \frac{\partial \mathcal{E}_h}{\partial x_i} - \int_{\Omega_h} \psi_{x,i}\lambda_h dxdy + \int_{\Omega_h} \rho_h \phi_i \partial_x \lambda_h dxdy =0, \nonumber\\
&\qquad\qquad\qquad\qquad\qquad\qquad\qquad\qquad \quad i=1,...,N; \\
\frac{\partial \tilde{\mathcal{R}}_h}{\partial v_{y,i}} &=\int_{\Omega_h} \rho_h v_{y,h} \phi_i dxdy + \frac{\partial \mathcal{E}_h}{\partial y_i} - \int_{\Omega_h} \psi_{y,i}\lambda_h dxdy + \int_{\Omega_h} \rho_h \phi_i \partial_y \lambda_h dxdy =0, \nonumber\\
&\qquad\qquad\qquad\qquad\qquad\qquad\qquad\qquad \quad  i=1,...,N; \\
\frac{\partial \tilde{\mathcal{R}}_h}{\partial \lambda_i} &= \int_{\Omega_h} \partial_t \rho_h \phi_i  + \rho_h (v_{x,h} \partial_x \phi_i + v_{y,h}\partial_y \phi_i) dxdy =0 ,\nonumber\\ &\qquad \qquad\qquad\qquad\qquad\qquad\qquad\qquad \quad  i=1,...,N_{in}.
\end{align}
The equation can be written in an algebraic form as
\begin{align*}
\b{M}\b{\lambda} &= \frac{\partial \mathcal{E}_h}{\partial \b{\rho}},\\
\b{D} \b{v_x} &= -\frac{\partial \mathcal{E}_h}{\partial \b{x}} +(\b{B}_x - \b{E}_x)^T \b{\lambda} ,\\
\b{D} \b{v_y} &=-\frac{\partial \mathcal{E}_h}{\partial \b{y}} +(\b{B}_y - \b{E}_y)^T \b{\lambda} ,\\
\b{A} \dot{\b{\rho}} &= -(\b{B}_x-\b{E}_x) \b{v_x} -(\b{B}_y-\b{E}_y) \b{v_y},
\end{align*}
where 
\begin{align*}
&M_{ij} = \int_{\Omega_h} \phi_i \phi_j dxdy ; \quad D_{ij} = \int_{\Omega_h} \rho_h \phi_i \phi_j dxdy; \\
&B_{x,ij} = \int_{\Omega_h} \phi_i \psi_{x,j} dxdy ; \quad B_{y,ij} = \int_{\Omega_h} \phi_i \psi_{y,j} dxdy ;\\
&E_{x,ij} = \int_{\Omega_h} \rho_h \partial_x \phi_i \phi_j^n dxdy ; \quad E_{y,ij} = \int_{\Omega_h} \rho_h \partial_y \phi_i \phi_j^n dxdy .
\end{align*}
Similarly to the one dimensional case, we can further discretize the time derivative in the above system by an explicit Euler scheme or an implicit linearized scheme. We neglect the details for simplicity in presentation.




\bibliographystyle{elsarticle-num}
\bibliography{ref}




\end{document}